\documentclass[20pt, oneside]{article}   	
\usepackage{geometry}                		
\usepackage{amsmath, tikz-cd}
\usepackage{amsthm}	
\usepackage{amssymb}
\usepackage{tikz}
\usepackage[implicit]{hyperref}
\usepackage[bottom]{footmisc}
\usepackage[affil-it]{authblk}
\usepackage[toc,page]{appendix}
\usepackage{algorithm}
\usepackage{algpseudocode}

\newtheorem{thm}{Theorem}[section]

\newtheorem{prop}[thm]{Proposition}

\newtheorem{lem}[thm]{Lemma}
\newtheorem{cor}[thm]{Corollary}
\newtheorem*{claim*}{Claim}

\newtheorem*{thm*}{Theorem}

\newtheorem{ex}[thm]{Example}
\newtheorem*{cx}{Counterexample}
\newtheorem{man}{Theorem}

\theoremstyle{definition}
\newtheorem*{rem*}{Remark}
\newtheorem{defi}[thm]{Definition}
\newtheorem{rem}[thm]{Remark}

\newcommand{\F}{\mathbb{F}}
\newcommand{\fl}{\mathfrak{l}}

\newcommand{\fa}{\mathfrak{a}}
\newcommand{\fb}{\mathfrak{b}}
\newcommand{\fp}{\mathfrak{p}}

\newcommand{\GL}{\mathrm{GL}}

\newcommand{\ff}{\mathfrak{f}}
\newcommand{\fL}{\mathfrak{L}}
\newcommand{\cO}{\mathcal{O}}

\newcommand{\xdownarrow}[1]{%
  {\left\downarrow\vbox to #1{}\right.\kern-\nulldelimiterspace}
}

\title{CM Drinfeld Modules, Self-isogenous Modular Polynomials, and Volcano Structure}
\author{Chien-Hua Chen \thanks{Electronic address: \texttt{cc45@aub.edu.lb}; ORCID: \texttt{0000-0003-3267-5603} ; Corresponding author}}
\affil{}

\begin{document}

\maketitle
\abstract
In this paper, we develop a view of self-isogenous modular polynomials and the $\fl$-cyclic isogeny graph for CM Drinfeld modules of arbitrary rank $r$. On the computational side, we give an explicit procedure to construct the  modular polynomial $\Phi_{J,\fa}(X,X)$ for Drinfeld modules of rank $r\geqslant 3$ with $\fa$ a prime ideal of $\F_q[T]$. When $\fa=(T)$, we provide an algorithm to compute  $\Phi_{J,\fa}(X,X)$; when $\fa=(T^2+T+1)$, we give an explicit degree bound on  $\Phi_{J,\fa}(X,X)$. On the structural side, we formulate a generalized $\fl$-cyclic volcano structure and prove that the generalized volcano appears in a component of the full $\fl$-cyclic isogeny graph for rank-$r$ Drinfeld modules with complex multiplication.

\section {Introduction}

On the modular polynomials arised from Drinfeld modules, the rank-2 case is classical: Bae \cite{B97} gave explicit constructions of Drinfeld modular polynomials, while Schweizer \cite{S95} analyzed their computational properties in the level-\((T)\) case. Later on, Breuer--R\"uck\cite{BR16} generalized to a formulation of modular polynomials arised from Drinfeld modules of arbitrary rank-$r$. However, it is unclear to formulate a generalization of bivariate modular polynomial $\Phi_\fl(X,Y)$ for Drinfeld modules of rank $r>2$. 

In the elliptic setting, computing the self-isogenous modular polynomial\(\Phi_{\ell}(X,X)\) is directly from the restriction $\Phi_{\ell}(X,Y)|_{X=Y}$. On the other hand, $\Phi_{\fl}(X,X)$ is related to computing Hilbert class polynomials \(H_{\mathcal O}(X)\) via CM theory, roots of \(H_{\mathcal O}\) parametrize CM \(j\)-invariants, and \(\Phi_{\ell}(X,X)\) detects those admitting an \(\ell\)-cyclic self-isogeny.
A computational breakthrough is Sutherland's \cite{AA10} CRT method for Hilbert class polynomials, which computes \(H_{\mathcal O}\)(X) modulo many primes by navigating \(\ell\)-isogeny graphs of elliptic curves over finite fields and then reconstructs \(H_{\mathcal O}\) over \(\mathbb{Z}\) by the Chinese Remainder Theorem. This provides an algorithmic bridge between isogeny graphs and modular polynomials. Inspired from the elliptic curve case, we study this phenomenon for Drinfeld modules of rank $r$, especially for the case $r>2$.

Motivated from \cite{BR16}, instead of defining an appropriate bivariate Drinfeld modular polynomial for rank $r>2$, we define the single variate Drinfeld modular polynomial $\Phi_{J,\fl}(X,X)$ directly. The polynomial was constructed by considering a polynomial over $\mathbb{C}_\infty$ whose roots are those self $\fl$-cyclic isogenus Drinfeld modules, up to isomorphic copy. Besides, inspired from the work of \cite{BR24}, we compute the $(T)$-cyclic self isogenous modular polynomial for arbitrary rank $r>2$, and conclude our computation as an algorithm (see Algorithm 1 at p.9). Then we generalize our computation to $\fa$-cyclic self-isogenous modular polynomials, wehere $\fa$ is a prime ideal of $\F_q[T]$. We show in Example \ref{deg2} that even in the simple situation where $\fa$ is generated by a degree-$2$ polynomial, the computation is much more complicated, and we can only deduce an upper bound for $\Phi_{J,\fa}(X,X)$. On the other hand, we prove that our formulation for $\Phi_{J,\fa}(X,X)$ still has a factorization into products of Hilbert class polynomials.

\begin{man}[Theorem \ref{hilbert}]

Fix a basic $J$-invariant for rank-$r$ Drinfeld modules, given a prime ideal $\mathfrak{a}\in A$ with monic generator $a$,

$$\Phi_{J,a}(X,X)=\prod_{\mathcal{O} \textrm{ imaginary $r$-degree order over $\F_q[T]$}}H_{\mathcal{O},J}(X)^{\gamma(\mathcal{O},a)},$$
where 
$$\begin{array}{ll}
\gamma(\mathcal{O},a)=&|\{ u\in \mathcal{O} \mid u \textrm{ primitive and }{\rm Nm}_{K/F}(u)=c\cdot a \textrm{ for some }c\in\F_q^*\}/\mathcal{O}^*|.
\end{array}
$$
Note that $u\in\mathcal{O}$ is primitive if $u\neq d\cdot \beta$ for some non-unit element $d\in A$ and $\beta \in \mathcal{O}$. Moreover, $\Phi_{J,a}(X,X)$ is invariant under ${\rm Gal}(\bar{F}/F)$-action, i.e. $\Phi_{J,a}(X,X)\in A[X]$.

\end{man}

For the isogeny graph side, in the classical setting for elliptic curves with nontrivial endomorphism rings, Kohel \cite{K96} firstly discovered volcano structure for isogeny graph for elliptic curves over finite fields, and hed explained how local endomorphism data at a prime $\ell$ stratifies
$\ell$-isogeny graph into a crater (maximal order) and layered levels beneath it, with a unique upward edge and controlled horizontal behavior on the crater. This idea has been extended to ordinary abelian varieties over finite fields in \cite{BJW17}, and later on generalized by the work of \cite{AMS25}. In another direction, Clark \cite{C22} considered elliptic curves over $\mathbb{C}$ with complex multiplication by a quadratic imaginary field $K$, and he showed that a volcano structure exists under a mild condition on endomorphism ring of CM elliptic curves and driscriminant of $K$. In the function field arithmetic, Caranay \cite{C18} studied the volcano structure for rank-$2$ Drinfeld modules over finite field, reinforcing the analogy between the number field and function field arithmetic. However, it is still unclear whether there are volcano structure lies in isogeny graph for Drinfeld modules of rank $r>2$ over finite fields, and Drinfeld modules over $\mathbb{C}_\infty$ of arbitrary rank $r$  with CM. This motivates us to consider the $\fl$-cyclic isogeny graph for CM Drinfeld modules, and it turns out that there is a component in the isogeny graph that has a ``generalized volcano structure''. Under Definition \ref{lev}, we have the following result:

\begin{man}[Theorem \ref{voc}]

\begin{enumerate}
\item[(i)]
In $\mathcal{G}_{\fl,K}^{\rm min, unif}$, vertex  $[\phi]$ with level $t\geqslant 1$ has only one ascending $\fl$-cyclic isogeny to a vertex $[\phi']$ of level $t-1$ and preserve the $\fl$-comprime factor $\ff_0$. Besides, there is no horizontal or descending edge.

\item[(ii)] For vertex $[\phi]\in \mathcal{G}_{\fl,K}^{\rm min, unif}$ with level $t=0$, the horizontal edge are $\fl$-cyclic isogenies $u:\phi\rightarrow \phi'$ with ${\rm End}(\phi')\cong \cO$, the number of horizontal edges is the number of primes $\fL$ of $\cO_K$ above $\fl$ with inertia degree $f(\fL\mid \fl)=1$ . Note that $[\phi]$ has no descending edges.

\end{enumerate}

\end{man}

We also show in Example \ref{r3case} the shape of volcano structure under a given split type of $\fl$ in $\cO_K$ for the rank-$3$ case.
\section*{Acknowledgement}

Part of this work was carried out during the author's visit to Prof. Florian Breuer at the Department of Mathematics, University of Newcastle. The author is grateful to the department for its warm hospitality and to Prof. Breuer for many insightful discussions. The author also thanks Prof. Fu-Tsun Wei for his valuable suggestions.

\section{Preliminaries}

Let $A=\F_q[T]$ be the polynomial ring over finite field with $q=p^e$ an odd prime power, $F=\F_q(T)$ be the fractional field of $A$, and  $K$ be a finite extension over $F$. Throughout this paper, ``$\log$'' refers to the logarithm with base $q$. We view $K$ as an $A$-field, which is a field equipped with a homomorphism $\gamma: A\rightarrow K$. The {\bf{$A$-characteristic}} of $K$ is defined to be the kernel of $\gamma$. When ${\rm ker}(\gamma)=0$ , we say $K$ is of generic characteristic  

\subsection{Drinfeld modules}

\begin{defi}

Let $K<x>:=\left\{\sum_{i=0}^{n}c_ix^{q^i}\mid c_i\in K\right\}$. Define $(K<x>,+,\circ)$ to be the ring of twisted $q$-polynomials with usual addition, and the multiplication is defined to be composition of $q$-polynomials.

\end{defi}

\begin{defi}
A Drinfeld $A$-module of rank $r$ over $K$ of generic characteristic is a ring homomorphism

$$\phi: A\rightarrow K<x>, \ \ a\mapsto \phi_a(x) $$
determined by
$$\phi_T(x)=Tx+g_1x^q+\cdots+g_rx^{q^r}.$$
\end{defi}

For an ideal $\mathfrak{a}=<a>$ of $A$, we may define the $\mathfrak{a}$-torsion of the Drinfeld module $\phi$ over $K$.

\begin{defi}
The $\mathfrak{a}$-torsion of a Drinfeld module $\phi$ over $K$ is defined to be
$$\phi[\mathfrak{a}]:=\left\{\textrm{ zeros of }\phi_a(x) \textrm{ in } \bar{K} \right\}\subset \bar{K}.$$
\end{defi}

Now we define the $A$-module structure on $\bar{K}$. For any elements $b\in A$ and $\alpha\in \bar{K}$. We define the $A$-action of $b$ on $\alpha$ via
$$b\cdot\alpha:=\phi_b(\alpha).$$
This gives $\bar{K}$ an $A$-module structure. And the $A$-module structure inherits to $\phi[\mathfrak{a}]$. As our Drinfeld module $\phi$ over $K$ has generic characteristic, we have the following proposition

\begin{prop}\label{prop0.2}
Let $\phi$ be a rank $r$ Drinfeld module over $K$ and $\mathfrak{a}$ a non-zero ideal of $A$,
 If $\phi$ has $A$-characteristic prime to $\mathfrak{a}$, then the $A/\mathfrak{a}$-module $\phi[\mathfrak{a}]$ is free of rank $r$
\end{prop}
\begin{proof}
See \cite{G96} Proposition 4.5.7.
\end{proof}

\begin{defi}
Let $\phi$ and $\psi$ be two rank-$r$ Drinfeld $A$-modules over $K$. A {{morphism}} $u:\phi\rightarrow \psi$ over $K$ is a twisted $q$-polynomial $u\in K<x>$ such that
$$u\phi_a=\psi_a u \text{ \rm for all } a\in A.$$ 
A non-zero morphism $u:\phi\rightarrow \psi$ is called an isogeny. A morphism $u:\phi\rightarrow \psi$ is called an {{isomorphism}} if its inverse exists. 
\end{defi}

Set ${\rm Hom}_K(\phi,\psi)$ to be the group of all morphisms $u:\phi\rightarrow \psi$ over $K$. We denote ${\rm End}_K(\phi)={\rm Hom}_K(\phi,\phi)$. For any field extension $L/K$, we define
$${\rm Hom}_L(\phi,\psi)=\{u\in L<x> \mid u\phi_a=\psi_a u  \text{ \rm for all } a\in A \}.$$
For $L=\bar{K}$, we omit subscripts and write

$${\rm Hom}(\phi,\psi):={\rm Hom}_{\bar{K}}(\phi,\psi) \text{ \rm and } {\rm End}(\phi):={\rm End}_{\bar{K}}(\phi)$$

\begin{rem}
Given an isogeny $u:\phi\rightarrow \psi$ between two rank-$r$ Drinfeld modules over $K$, let $c,d$ be a non-zero elements in $\bar{K}$. By taking conjugation action of $c$ on $\psi$, one gets an isogeny $c^{-1}u$ from $\phi$ and $c^{-1}\psi c$. Similarly, $u d$ is an isogeny from $d^{-1}\phi d$ to $\psi$. Therefore, we say an isogeny of the form $c^{-1}uc:c^{-1}\phi c\rightarrow c^{-1}\psi c$ is an {\bf isomorphic copy} of $u:\phi \rightarrow \psi$.
\end{rem}

\begin{defi}
We call a finite extension $L/F$ {{imaginary}} if there is a unique place $\tilde{\infty}$ of $L$ over $\infty$.

\end{defi}

\begin{defi}
\ \\
\begin{enumerate}
\item[$\bullet$]The composition of morphisms makes ${\rm End}_L(\phi)$ into a subring of $L<x>$, called the {\bf{endomorphism ring}} of $\phi$ over $L$. For any rank-$r$ Drinfeld module $\phi$ over $K$ with ${\rm End}(\phi)=A$, we say that $\phi$ does not have complex multiplication.

\item[$\bullet$] Let $K/F$ be an imaginary degree-$r$ extension with ring of integers $\cO_K$, a rank-$r$ Drinfeld module $\phi$ is said to be CM by $K$ if ${\rm End}(\phi)$ is isomorphic to an order of $K$. And we say $\phi$ is CM by an order $\cO$ of $K$ if ${\rm End}(\phi)\cong \cO$.
\end{enumerate}
\end{defi}

\begin{defi}
Let $f:\phi\rightarrow \psi$ be an isogeny of Drinfeld modules over $K$ of rank $r$, we define the degree of $f$ to be
$$\deg f:=\# {\rm{ker}}(f).$$

\end{defi}

\begin{prop}\label{dual}
Let $f:\phi\rightarrow \psi$ be an isogeny of Drinfeld modules over $K$ of rank $r$. There exists a dual isogeny $\hat{f}:\psi\rightarrow \phi$ such that $$f\circ\hat{f}=\psi_a \textrm{ and } \hat{f}\circ f=\phi_a.$$
Here  $0\neq a\in A$ is an element of minimal $T$-degree such that ${\rm ker}(f)\subset \phi[a]$. 

\end{prop}
\begin{proof}
See \cite{G96} Proposition 4.7.13 and Corollary 4.7.14.

\end{proof}

\subsection{CM uniformization and ideal-isogeny construction}

Let $\phi$ be a rank-$r$ Drinfeld modules with CM by an imaginary degree-$r$ order $\cO\subset\cO_K$. Fix an embedding $K\hookrightarrow \mathbb{C}_\infty$, one can view $\phi$ as a rank-$1$ Drinfeld $\cO$-module, hence there is a fractional  $\cO$-ideal $I\subset K$ such that the following commutative diagram holds for $a\in \cO$:

\[\begin{tikzcd}
	{\mathbb{C}_\infty/I} & {\mathbb{C}_\infty/I} \\
	{\mathbb{C}_\infty} & {\mathbb{C}_\infty}
	\arrow["a", from=1-1, to=1-2]
	\arrow["{e_I}"', from=1-1, to=2-1]
	\arrow["{e_I}", from=1-2, to=2-2]
	\arrow["{\phi_{I,a}}"', from=2-1, to=2-2]
\end{tikzcd}\]

Here the top right arrow is multiplication-by-$a$ map, $e_I$ is the Drinfeld exponential map, and $\phi_{I,a}=\phi_a$. Besides, the endomorphism ring of $\phi_I$ is ${\rm End}(\phi_I)=\{x\in K \mid xI\subset I\}=\cO$, and an isomorphism between rank-$1$ Drinfeld $\cO$-modules $\phi_{I}\cong\phi_{I'}$ makes the corresponding fractional ideal $I$ and $I'$ satisfy $I=uI'$ for some $u\in K^*$.

Let $\fa$ be an $\cO$-ideal, set $\fa^{-1}:=\{x\in K\mid x\fa\subset \cO\}$ and consider the fractional ideal $\fa^{-1}I$. The correspond Drinfeld module $\phi_{\fa^{-1}I}$ induces an isogeny $\iota_\fa:\phi\rightarrow\phi_{\fa^{-1}I},$ hence we define $\phi_\fa:=\phi_{\fa^{-1}I}$. Moreover, ${\rm ker}(\iota_\fa)=\phi[\fa]\cong\cO/\fa$ as $\cO$-modules, and the endomorphism ring of $\phi_\fa$ is $${\rm End}(\phi_\fa)=\{x\in K\mid x(\fa^{-1}I)\subset\fa^{-1}I\}=\{x\in K\mid x\fa\subset\fa\}=(\fa:\fa).$$

On the other hand, we define an action of elements $[\fb]$ in Picard group ${\rm Pic}(\cO)$ on the set $$\left\{[\phi_I]: \mathbb{C}_\infty\textrm{-isomorphism classes of rank-$1$ Drinfeld $\cO$-modules } \right\}$$ via $[\fb][\phi_I]:=[\phi_{\fb^{-1}I}]$. One can check that the action is well-defined.

\subsection{Modular polynomials}

\begin{defi}
We take the definition of $J$-invariants of higher rank Drinfeld modules introduced by Potemine \cite{P98}. Given a rank-$r$ Drinfeld module $\phi$ over an arbitrary $A$-field $K$, whose $A$-characteristic is equal to zero, with
$$\phi_T=T+g_1\tau+\cdots+g_{r-1}\tau^{r-1}+g_r\tau^r, \textrm{ where }g_i\in K \textrm{ and }g_r\in K^*.$$
We define a basic $J$-invariant of $\phi$ to be
$$J^{(\delta_1,\cdots,\delta_{r-1})}(\phi)=\frac{g_1^{\delta_1}\cdots g_{r-1}^{\delta_{r-1}}}{g_r^{\delta_r}},$$
where $\delta_i$'s satisfy the following two conditions:
\begin{equation}
 \delta_1(q-1)+\delta_2(q^2-1)+\cdots+\delta_{r-1}(q^{r-1}-1)=\delta_r(q^r-1). \tag{1}
\end{equation}

\begin{equation}
0\leqslant \delta_i \leqslant \frac{q^r-1}{q^{{\rm{g.c.d.}}(i,r)}-1} \textrm{ for all }1\leqslant i\leqslant r-1;\ \ \text{ g.c.d.$(\delta_1,\cdots,\delta_r)=1.$} \tag{2}
\end{equation}

\end{defi}
\begin{defi}\label{basicj}
We denote by $\left\{  J^{(\delta_1,\cdots, \delta_{r-1})}    \right\}$ the set of all basic $J$-invariants.

\end{defi}

\begin{prop}
$$M^r(1)={\rm Spec} A\left[\left\{  J^{(\delta_1,\cdots, \delta_{r-1})}    \right\}\right]$$ is the coarse moduli scheme of Drinfeld $A$-modules of rank $r$. Here $A\left[\left\{  J^{(\delta_1,\cdots, \delta_{r-1})}    \right\}\right]$ is the ring generated by  $\left\{  J^{(\delta_1,\cdots, \delta_{r-1})}    \right\}$ over $A$. Any $0\neq J\in A\left[\left\{  J^{(\delta_1,\cdots, \delta_{r-1})}    \right\}\right]$ is called an invariant
\end{prop}

\begin{proof}
See \cite{P98}, Theorem 3.1.

\end{proof}

From Proposition 2.6 in \cite{BR16}, we can define modular polynomials as the following:

\begin{defi}
Any isogeny $f$ of a Drinfeld $A$-module $\phi$ defined over an $A$-field $K$ satisfying ${\rm ker} f\subset \phi[N]$ is called an $N$-isogeny. Let $K_N$ be the splitting field of $\phi[N]$ over $K$, and $R_N$ be the integral closure of $\mathcal{O}_K$ in $K_N$, then such an isogeny $f$ is defined over $R_N$.

\end{defi}
Let $I_N$ be the set of all monic $N$-isogenies $f\in R_N(X)$, and $J$ be an invariant. Define the full modular polynomial of level $N$ associated to $J$ by $$\Phi_{J,N}(X):=\prod_{f\in I_N}(X-J(\phi^{(f)}))\in R_N[X].$$

Here $\phi^{(f)}$ is the Drinfeld module corresponding to the monic isogeny $f:\phi\rightarrow \phi^{(f)}$.

\begin{defi}

Let $H\subset \phi[N]\cong(A/NA)^r$ be an $A$-submodule. Any isogeny $f\in I_N$ where ${\rm ker} f$ is a $\GL_r(A/NA)$-orbit of $H$ is called an isogeny of type $H$. Let $J$ be an invariant, define the modular polynomial of type $H$ associated to $J$ by $$\Phi_{J,H}(X):=\prod_{f\in I_N \textrm{ of type }H}(X-J(\phi^{(f)}))\in R_N[X].$$

\end{defi}

On the other hand, we define a bivariate modular polynomial for Drinfeld modules of arbitrary rank.
\begin{defi}
Given a rank $r$ and $H\subset (A/NA)^r$ a submodule. Any monic isogeny $f:\phi\rightarrow \psi$ between rank-$r$ Drinfeld modules defined over $\mathbb{C}_\infty$ with ${\rm ker} f$ is a $\GL_r(A/NA)$-orbit of $H$ is called an isogeny of type $H$. 

Let $J$ be an invariant, define $\Phi_{J,H}(X,X)\in \mathbb{C}_\infty[X]$ to be a polynomial such that $\phi_{J,H}(a,a)=0 \textrm{ if and only if }$
\textrm{$\exists$ Drinfeld modules $\phi$ over $\mathbb{C}_\infty$ with $J(\phi)=a$, a monic isogeny $f:\phi\rightarrow \phi$ of type $H$ , and up to isomorphic copy $c^{-1}fd$ of $f$.} 

\end{defi}

\begin{defi}
If a $N$-isogeny $f:\phi\rightarrow \psi$ has ${\rm ker} f\cong A/N$ we say $f$ is a cyclic isogeny. As $N=(n)$ is an ideal in $A$, we define
$$\Phi_{J,N}(X,X)=\Phi_{J,n}(X,X):=\Phi_{J,A/N}(X,X).$$

\end{defi}

\section{Computation on $\mathfrak{a}$-cyclic modular polynomial}

\subsection{Direct computation on $\Phi_{J,T}(X,X)$}

Fix a rank $r$, consider a $T$- cyclic self isogeny $f(X)=a_0X+X^q$ of a Drinfeld module $\phi_T(X)=TX+g_1X^q+\cdots+g_{r-1}X^{q^{r-1}}+\Delta X^{q^r}$, with dual isogeny $\hat{f}=b_0X+b_1X^q+\cdots+b_{r-2}X^{q^{r-2}}+\Delta X^{q^{r-1}}$ such that
\begin{equation}\tag{$\star$}
f\circ \hat{f}=\hat{f}\circ f=\phi_T(X).
\end{equation}
We compute
\begin{enumerate}
\item $$f\circ \hat{f}=a_0b_0X+\sum_{i=1}^{r-2}[a_0b_i+b_{i-1}^q]X^{q^i}+[a_0\Delta+b_{r-2}^q]X^{q^{r-1}}+\Delta^qX^{q^r}.$$

\item $$\hat{f}\circ f=b_0a_0X+\sum_{i=1}^{r-2}[b_ia_0^{q^i}+b_{i-1}]X^{q^i}+[\Delta a_0^{q^{r-1}}+b_{r-2}]X^{q^{r-1}}+\Delta X^{q^r}.$$
\end{enumerate}

Therefore, we have the following from comparing coefficients

\begin{equation}\label{1}
    \begin{cases}
      a_0b_0=b_0a_0\\
     a_0b_i+b_{i-1}^q=b_ia_0^{q^i}+b_{i-1},{\quad 1\leqslant i\leqslant r-2}\\
     a_0\Delta+b_{r-2}^q=\Delta a_0^{q^{r-1}}+b_{r-2}\\
     \Delta^q=\Delta
    \end{cases}\,.
\end{equation}

From the system (1) and the equality ($\star$), one can deduce that

\begin{equation}\label{2}
    \begin{cases}
      b_0=T/a_0\\
    b_i=\frac{1}{a_0^{q^i}-a_0}[b_{i-1}^q-b_{i-1}],{\quad 1\leqslant i\leqslant r-2}\\
     \Delta=\frac{1}{a_0^{q^{r-1}}-a_0}[b_{r-2}^q-b_{r-2}]\\
     \Delta^{q-1}=1
    \end{cases}\,.
\end{equation}

In conclusion, we have the following Proposition:

\begin{prop}\label{Tcyclic}
Given a Drinfeld module $\phi_T(X)=TX+g_1X^q+\cdots+g_{r-1}X^{q^{r-1}}+\Delta X^{q^r}$, over $K$ of $A$-characteristic $0$ , with $\Delta\in\F_q^*$. If $\phi$ has a $T$-cyclic self isogeny $f(X)=a_0X+X^q$, then there exists a polynomial $g(X,\Delta)\in A[X]$ such that $a_0$ is a root of $g(X,\Delta)$. \end{prop}

\begin{proof}
One can obtain the polynomial $g(X,\Delta)$ from system (\ref{2}) by iterating $$b_i=\frac{1}{a_0^{q^i}-a_0}[b_{i-1}^q-b_{i-1}]$$
 into $\Delta=\frac{1}{a_0^{q^{r-1}}-a_0}[b_{r-2}^q-b_{r-2}]$, then rationalize the $a_0$-terms at denominator. 
\end{proof}

\begin{prop}\label{a0}
$a_0$ does not belong to the constant field $\bar{\F}_q$. Moreover, among all $\Delta\in\F_q^*$, roots of $g(X,\Delta)$ would contribute non-isomorphic $T$-cyclic isogenies.
\end{prop}

\begin{proof}

Suppose contrary that $a_0\in \bar{\F}_q$, then $b_0=T/a_0$ belongs to $\bar{F}_q[T]$ and has $T$-degree equal to $1$. From system (\ref{1}) we have $$b_i(a_0^{q^i}-a_0)=b_{i-1}^q-b_{i-1} \textrm{ for }1\leqslant i\leqslant r-2.$$
Thus $b_i\in\bar{\F}_q[T]$ has $T$-degree equal to $q^i$ for $1\leqslant i\leqslant r-2$. On the other hand, we compare both sides of the equality 
$$\Delta(a_0^{q^{r-1}}-a_0)=b_{r-2}^q-b_{r-2}.$$
The left hand side belongs to $\bar{\F}_q$, while the right hand side has $T$-degree equal to $q^{r-1}$, a contradiction. 

On the other hand, for a solution $a_0$ of $g(X,\Delta)$, it corresponds to the isogeny $u(X)=a_0X+X^q$. If there is no other distinct root $a_0'$ of $g(X,\Delta)$ with $u'=c^{-1}uc$ for some $c\in \mathbb{C}_\infty^*$.

\end{proof}

We compute an easy example below as a rank-$3$ case:

  \begin{ex}
 When $r=3 \textrm{ and } q=2$, combine the system (\ref{1}) with equation ($\star$) becomes
 \begin{equation}\label{1'}
    \begin{cases}
      a_0b_0=b_0a_0=T\\
      a_0b_1+b_{0}^2=b_1a_0^{2}+b_{0}=g_1\\
      a_0\Delta+b_{1}^2=\Delta a_0^{2}+b_{1}=g_2\\
     \Delta^2=\Delta
    \end{cases}\,.
\end{equation}
 
Hence one can reduce to the following

 \begin{equation}\label{1'}
    \begin{cases}
      b_0=T/a_0\\
      b_1=\frac{1}{a_0^2+a_0}\left(\frac{T^2+a_0T}{a_0^2}\right)\\
      \Delta=1\\
      b_1^2+b_1=a_0^4+a_0
    \end{cases}\,.
\end{equation}
 
 As $b_1^2=\frac{1}{a_0^4+a_0^2}\left(\frac{T^4+a_0^2T^2}{a_0^4}\right)$, the equation $b_1^2+b_1=a_0^4+a_0$ becomes
 $$\frac{T^4+a_0^2T^2}{a_0^8+a_0^6}+\frac{T^2+a_0T}{a_0^4+a_0^3}+a_0^4+a_0=0,$$
 $$\Rightarrow (T^4+a_0^2T^2)+(a_0^4+a_0^3)(T^2+a_0T)+(a_0^{12}+a_0^{10})+(a_0^9+a_0^7)=0.$$
 Therefore, $a_0$ is a root of the polynomial 
 $$
 \begin{array}{lll}
 g(\Delta,X)&=&X^{12}+X^{10}+X^9+X^7+TX^5+(T^2+T)X^4+T^2X^3+T^2X^2+T^4\\
 \ \\
 &=&(X^3+T)(X^3+X+T)(X^3+X^2+T)(X^3+X^2+X+T).
 \end{array}
 $$
 
There are $12$ distinct roots $a_0$ for the polynomial $g(\Delta,X)$. For each choice of $a_0$, one can recover a Drinfeld module:
$$\phi_{a_0,T}(X)=TX+g_1X^q+g_2X^{q^2}+X^{q^3},$$
where 
$$g_1=\frac{T(T+1)}{a_0(a_0+1)},\textrm{ and }g_2=a_0+\frac{T^4+a_0^2T^2}{a_0^8+a_0^6}$$
 
Pick a $J$-invariant $J^{(1,2)}(\phi)=\frac{g_1g_2^2}{\Delta}$, then 

$$J^{(1,2)}(\phi_{a_0})=\left[\frac{T^2+T}{a_0^2+a_0}\right]\left[a_0^2+\frac{T^8+a_0^4T^4}{a_0^{16}+a_0^{12}} \right]=\frac{T(T+1)(a_0^7+a_0^5+1)^2}{a_0^9(a_0+1)^5}$$ Hence we obtain
 $$\Phi_{J^{(1,2)},T}(X,X)=\prod_{a_{0,\sigma} \textrm{ Galois conjugate of }g(\Delta,X)}\left(X-\frac{T(T+1)(a_{0,\sigma}^7+a_{0,\sigma}^5+1)^2}{a_{0,\sigma}^9(a_{0,\sigma}+1)^5} \right).$$
 
 \end{ex}

As a conclusion, we write an algorithm in next page to find $\Phi_{J,T}(X,X)$ using the direct computation method.

\begin{algorithm}

\caption{Construct rank $r$ self-isogenous modular polynomial $\Phi_{J,T}(X,X)$}\label{alg:cap}
For simplicity, we consider the case $J$ is a basic $J$-invariant.
\begin{algorithmic}

\item[\bf Input:] $T$ (symbol/parameter), $r$ (rank), $q=p^e$ (prime power), $J$ (a basic $J$-invariant)
\item[\bf Output:] List of roots $a_{0,\sigma}$, and $\Phi_{T,J}(X,X)$

\item[\bf Step 1.] Setup\\
\quad $\bullet$ Define unknowns $a_0, b_0, b_1,\cdots, b_{r-1}, \Delta, g_1,\cdots, g_{r-1}$\\
\quad $\bullet$ Fix a basic $J$-invariant for rank-$r$ Drinfeld modules:\\
\quad\quad (a) $J=J^{(\delta_1,\cdots,\delta_{r-1})}=\frac{g_1^{\delta_1}\cdot\cdots\cdot g_{r-1}^{\delta_{r-1}}}{\Delta^{\delta_r}}$\\
\quad\quad (b) $\delta_i$'s are positive integers with $\sum_{i=1}^{r-1}\delta_i(q^i-1)=\delta_r(q^r-1)$\\
\quad\quad (c) $gcd(\delta_1,\cdots,\delta_r)=1$, and $0\leqslant \delta_i \leqslant \frac{q^r-1}{q^{{\rm{g.c.d.}}(i,r)}-1}$ for $1\leqslant i\leqslant r-1$\\

\item[\bf Step 2.] Impose relations:\\
\quad (i) $b_0=T/a_0$\\
\quad (ii) For $i=1,\cdots, r-2,$ compute $b_i$ recursively by $a_0b_i+b_{i-1}^q=b_ia_0^{q^i}+b_{i-1}$\\
\quad (iii) $\Delta\in \F_q^*$\\
\quad (iv) $a_0\Delta+b_{r-2}^q=\Delta a_0^{q^{r-1}}+b_{r-2}$\\

\item[\bf Step 3.]  Further computation:\\
\quad (1) For $1\leqslant i\leqslant r-2$, solve iteratively from Step 2(ii) for $b_i$ in terms of $a_0$;\\ 
\quad (2) Assign a value $\Delta\in \F_q^*$, use (1) to rationalize the rational equation (iv) in terms of $a_0$. \\
\quad (3) Obtain from (2) a polynomial $g(\Delta,X)\in\F_q(T)[X]$ which has $a_0$ as its roots.\\

\item[\bf Step 4.] Construct $\phi_{a_0}$\\
For each $\Delta\in \F_q^*$, run through non-constant roots $a_0$ of $g(\Delta,X)$:\\
\quad (a) For $1\leqslant i\leqslant r-2$, express in terms of $a_0$ recursively $g_i=a_0b_i+b_{i-1}^q$;\  \ $g_{r-1}=a_0\Delta+b_{r-2}^q$.\\
\quad (b) List $\phi_{a_0}=TX+g_1X^q+\cdots+g_{r-1}X^{q^{r-1}}+\Delta X^{q^r}$\\

\item[\bf Step 5.] Compute $\Phi_{J,T}(X,X)$\\

$$\Phi_{J,T}(X,X)=\prod_{a_0: \textrm{ non-constant roots of } g(\Delta,X)}\left(X-J(\phi_{a_0})\right)$$

\end{algorithmic} 
\end{algorithm}

\newpage
\subsection{General prime ideal $\mathfrak{a}$}

 \begin{prop}
 Let $r$ be a positive integer, and $a\in \F_q[T]$ be a monic prime of $T$-degree $d$. If $u$ is a self $a$ isogeny of $A$-dimension $m$ for a rank-$r$ Drinfeld $A$-module $\phi$ over $\overline{\F_q(T)}$. Let $\theta=a$, and define a Drinfel $\F_q[\theta]$-module $\psi_\theta=\phi_a$ of rank $rd$  over $\overline{\F_q(\theta)}=\overline{\F_q(T)}$. Then $u$ is a self $\theta$-isogeny of $\F_q[\theta]$-dimension $dm$ for the Drinfeld module $\psi$.
 \end{prop}
 
 \begin{proof}
 
 Fix a monic prime $a\in \F_q[T]$ of $T$-degree equal to $k$. Let $u$ be a self $a$-isogeny of $A$-dimension $d$. In other words, we have the following properties:
 \begin{enumerate}
 \item[$\bullet$] $u\phi_T=\phi_T u$ 
 \item[$\bullet$] Let $\hat{u}$ be the dual isogeny of $u$, we have $u\circ\hat{u}=\hat{u}\circ u=\phi_a$
 \item[$\bullet$] $ \left(A/(a)\right)^m\cong{\rm ker}(u)\subset \phi[a]$ as an $A$-module.
 \end{enumerate}

 Therefore, for the Drinfeld $\F_q[\theta]$-module $\psi_\theta:=\phi_a$, we still have the properties $u\psi_\theta=\psi_\theta u$, and $u\circ \hat{u}=\hat{u}\circ u=\psi_\theta$. Hence $u:\psi\rightarrow\psi$ is a self $\theta$-isogeny for the Drinfeld $\F_q[\theta]$-module $\psi$. As $\# {\rm ker}(u)=q^{dm}$, we can deduce that ${\rm ker}(u)\cong \left(\F_q[\theta]/(\theta)\right)^{dm}$ as $\F_q[\theta]$-modules

 \end{proof}
 \begin{cor}
 All cyclic, self $a$-isogeny for rank-$r$ Drinfeld modules can be effectively computed from self $T$-isogeny for rank $rd$-Drinfeld modules, where $d= \deg_T(a)$.
 
 \end{cor}

 \begin{ex}\label{deg2}
 
 Let $r=3, q=p^e$ be a prime power with $p>3$, and $a\in\F_q[T]$  be a monic prime of degree $2$. In this example we show that direct computation on $\Phi_{a,J}(X,X)$ becomes complicated. 
 
 Instead of finding the explicit polynomial, we give an upper bound on the degree of $\Phi_{a,J}(X,X)$ depending only on the choice of $q$. Consider a rank-$6$ Drinfeld $\F_q[\theta]$-module $\psi_\theta(X)=\theta X+g_1X^q+\cdots+g_5X^{q^5}+g_6 X^{q^6}$ with a self $\theta$-isogeny $f=a_0X+a_1X^q+X^{q^2}$ of  $\F_q[\theta]$-dimension $2$. Set the dual isogeny of $f$ to be $\hat{f}(X)=b_0X+b_1X^q+b_2X^{q^2}+b_3X^{q^3}+g_6X^{q^4}$.
 
 From the property $f\circ \hat{f}=\hat{f}\circ f=\psi_\theta$, one can deduce by comparing coefficients that
 
 \begin{equation}\label{5}
 	\begin{cases}
		b_0=\theta/a_0\\
		a_0b_1+a_1b_0^q=b_0a_1+b_1a_0^q=g_1\\
		a_0b_2+a_1b_1^q+b_0^{q^2}=b_0+b_1a_1^q+b_2a_0^{q^2}=g_2\\
		a_0b_3+a_1b_2^q+b_1^{q^2}=b_1+b_2a_1^{q^2}+b_3a_0^{q^3}=g_3\\
		a_0g_6+a_1b_3^q+b_2^{q^2}=b_2+b_3a_1^{q^3}+g_6a_0^{q^4}=g_4\\
		a_1g_6^q+b_3^{q^2}=g_6a_1^{q^4}+b_3=g_5\\
		g_6\in\F_{q^2}^*
	\end{cases}
 \end{equation}
  
    \begin{lem}\label{a0}
  \begin{enumerate}
 \item[(i)] $a_0$ does not belong to $\F_q^*$
 \item[(ii)] $a_0$ does not belong to $\F_{q^2}^*-\F_q^*$.
 \item[(iii)] If $a_0\in \F_{q^3}^*-\F_{q}^*$, then $a_1$ satisfies a polynomial over $\F_{q^3}[\theta][X]$ of degree $q^2+q+1$.
 \item[(iv)] If $a_0\in \F_{q^4}^*-\F_{q^2}^*$, then $a_1$ satisfies a polynomial over $\F_{q^4}[\theta][X]$ of degree $q^4+q^2+q+1$.
 \end{enumerate}
 \end{lem}
 \begin{proof}
 
 \begin{enumerate}
 
 \item[(i)]  If $a_0\in\F_q$, then either $a_1=0$ or $b_0^q-b_0=0$. But $a_0b_0=\theta$ forces $b_0$ is not a constant. Hence we have $a_1=0$. However, from system (5) we get $b_0^{q^2}-b_0=0$, a contradiction.
 
 \item[(ii)] If $a_0\in \F_{q^2}^*-\F_q^*$, get $a_1b_1^q-a_1^qb_1+(b_0^{q^2}-b_0)=0$. Together with $b_1=\frac{a_1(b_0^q-b_0)}{a_0^q-a_0}$, one can deduce that $a_1^{q+1}=a_0^q-a_0$. Thus $a_1^{q^2+q}=-a_1^{q+1}$, which means $a_1^{q^2-1}=-1$. Hence we have $a_1\in \F_{q^4}^*$. However, by comparing highest $\theta$-degree terms of $a_1b_1^q-a_1^qb_1+(b_0^{q^2}-b_0)=0$, one can deduce that $$\left(\frac{a_1}{(a_0^q-a_0)a_0^q}\right)^q=\frac{-1}{a_0^{q^2}}.$$ 
 Thus $a_1^q=a_0^q-a_0$, a contradiction.
 \item[(iii)] If $a_0\in\F_{q^3}^*-\F_q^*$, get 
 \begin{equation}\label{iii}
 \begin{cases}
 b_0=\theta/a_0\\
		
b_1=\frac{a_1(b_0^q-b_0)}{a_0^q-a_0}\\ 

b_2=\frac{a_1b_1^q-b_1a_1^{q}+(b_0^{q^2}-b_0)}{a_0^{q^2}-a_0}\\

 a_1b_2^q-b_2a_1^{q^2}+(b_1^q-b_1)=0 \\
 
 g_6(a_0^q-a_0)=a_1b_3^q-b_3a_1^{q^3}+(b_2^{q^2}-b_2)\\
 
 g_6a_1^{q^4}-a_1g_6^q=b_3^{q^2}-b_3

 \end{cases}.
 \end{equation}
 From here, one can deduce that 
 $$a_1^{q^2+q+1}(\frac{N_1^q-N_1D_1^{q-1}}{D_1^q})+a_1(b_0^q-b_0)\frac{N_2}{D_2}=b_1-b_1^q,$$
 
 where $N_1=(b_0^q-b_0)^q-(b_0^q-b_0)(a_0^q-a_0)^{q-1}$, $N_2=(b_0^{q^2}-b_0)^{q-1}-a_1^{q^2-1}(a_0^{q^2}-a_0)^q$, 
 $D_1=(a_0^{q^2}-a_0)(a_0^q-a_0)^q$, and $D_2=(a_0^{q^2}-a_0)^q$. Thus $a_1$ satisfies a polynomial over $\F_{q^3}[\theta][X]$ of degree $q^2+q+1$.

 \item[(iv)]  If $a_0\in \F_{q^4}^*-\F_{q^2}^*$, then $a_1b_3^q-b_3a_1^{q^3}+(b_2^{q^2}-b_2)=0$. Together with 
 
  \begin{equation*}
 	\begin{cases}
		b_0=\theta/a_0\\
		b_1=\frac{a_1(b_0^q-b_0)}{a_0^q-a_0}\\
		b_2=\frac{a_1b_1^q-b_1a_1^{q}+(b_0^{q^2}-b_0)}{a_0^{q^2}-a_0}\\
		b_3=\frac{a_1b_2^q-b_2a_1^{q^2}+(b_1^q-b_1)}{a_0^{q^3}-a_0}
	\end{cases}
 ,\end{equation*}
 one can deduce that $a_1$ satisfies a polynomial over $\F_{q^4}[\theta][X]$ of degree $q^4+q^2+q+1$.

 \end{enumerate}
 
 \end{proof}

  If $a_0$ does not lie in cases of Lemma \ref{a0}, one can get the following expressions:
  
  \begin{equation}\label{6}
 	\begin{cases}
		b_0=\theta/a_0\\
		\ \\
		b_1=\frac{a_1(b_0^q-b_0)}{a_0^q-a_0}\\
		\ \\
		b_2=\frac{a_1b_1^q-b_1a_1^{q}+(b_0^{q^2}-b_0)}{a_0^{q^2}-a_0}\\
		\ \\
		b_3=\frac{a_1b_2^q-b_2a_1^{q^2}+(b_1^q-b_1)}{a_0^{q^3}-a_0}\\
		\ \\
		g_6=\frac{a_1b_3^q-b_3a_1^{q^3}+(b_2^{q^2}-b_2)}{a_0^{q^4}-a_0}\in \F_{q^2}^*
	\end{cases}
 \end{equation}

\begin{prop} 
 For $a_0$ not in the cases of Lemma \ref{a0}, we can express $b_1,b_2, b_3$ in terms of $a_0$ and $a_1$. As a conclusion, given $g_6\in\F_{q^2}^*$, the number of solutions $\#(a_0,a_1)$ satisfy system (\ref{5}) has $$\#(a_0,a_1)\leqslant (6q^4-q^3)5q^5=30q^9-5q^8$$
 \end{prop}
 
 \begin{proof}
 From system (\ref{6}), we have
  \begin{equation*}
 \begin{cases}
 b_1=\frac{a_1(\theta^q-a_0^{q-1}\theta)}{(a_0^q-a_0)a_0^q}\\
 \ \\
 b_2=\frac{a_1^{q+1}(b_0^q-b_0)\left[(b_0^q-b_0)^{q-1}-(a_0^q-a_0)^{q-1}\right]}{(a_0^{q^2}-a_0)(a_0^{q^2}-a_0^q)}+\frac{b_0^{q^2}-b_0}{a_0^{q^2}-a_0}\\
 \ \\
 b_3=\frac{N}{D} 
 \end{cases}
 ,\end{equation*}
 where 
$$
\begin{array}{ll}
D=&a_0^{2q^3-q^2}(a_0^{q^3}-a_0)(a_0^{q^2}-a_0)^q(a_0^q-a_0)^{q^2}\\
N=&a_1^{q^2+q+1}(N_1^q-N_1D_1^{q-1})+a_1(\theta^q-a_0^{q-1}\theta)N_2a_0^{q^3-q}(a_0^{q^2}-a_0)^{q^2}\\
    & +a_1^{2q}(\theta^q-a_0^{q-1}\theta)^qa_0^{2q^3-2q^2}(a_0^{q^2}-a_0)^q(a_0^q-a_0)^{q^2-q}\\
    & -a_0^{2q^3-q^2-q}(a_0^{q^2}-a_0)^q(a_0^q-a_0)^{q^2-1}\\
    
N_1=&(\theta^q-a_0^{q-1}\theta)\left[(\theta^q-a_0^{q-1}\theta)^{q-1}-(a_0^{2q}-a_0^{q+1})^{q-1} \right]\\
D_1=&a_0^{q^2}(a_0^{q^2}-a_0)(a_0^q-a_0)^q\\
N_2=&(\theta^{q^2}-a_0^{q^2-1}\theta)^{q-1}-a_0^{q^2(q-1)}a_1^{q^2-1}(a_0^{q^2}-a_0)^{q-1}
\end{array}
$$

 Note that $\deg_{a_0}(N)=3q^3-2q, \deg_{a_1}(N)=q^2+q+1; \deg_{a_0}(D)=5q^3-q^2, \deg_{a_1}(D)=0$. From equations $g_6=\frac{a_1b_3^q-b_3a_1^{q^3}+(b_2^{q^2}-b_2)}{a_0^{q^4}-a_0}$,  one can deduce that

$$
 \begin{array}{ll}
D^q(a_0^{q^4}-a_0)g_6=&a_1N^q-D^{q-1}Na_1^{q^3}+a_0^{q^4-q^3}(a_0^{q^4}-a_0^q)\tilde{N}^{q^2}\\
&-a_0^{2q^4-q^3-q^2}(a_0^{q^3}-a_0)^q(a_0^{q^2}-a_0)^{q^2-1}(a_0^q-a_0)^{q^3-q}\tilde{N}
 \end{array}
 ,$$
 where $b_2={\tilde{N}}/{\tilde{D}}$. Hence $(a_0,a_1)$ satisfies a bivariate polynomial of total degree $6q^4-q^3$
 
 From $g_6a_1^{q^4}-a_1g_6^q=b_3^{q^2}-b_3$, one can get
 
 $$D^{q^2}(g_6a_1^{q^4}-a_1g_6^q)=N^{q^2}-ND^{q^2-1}.$$
 Hence given $g_6\in\F_{q^2}^*$, $(a_0,a_1)$ satisfies a bivariate polynomial of total degree $5q^5$.
 
 Therefore, for every $g_6\in\F_{q^2}^*$, the number of solutions $(a_0,a_1)$ has the property that
 $$\#(a_0,a_1)\leqslant (6q^4-q^3)5q^5=30q^9-5q^8.$$
\end{proof}

\begin{prop}

When $\mathfrak{a}=T^2+T+1$, the number of Drinfeld $A$-modules $\phi$ with self $\mathfrak{a}$-cyclic isogeny is upper bounded by  $\mathcal{N}_q=(q^3+1)(q^2-1)(30q^{15}-4q^{14}+q^{12}+2q^{11}-q^{10}-2q^9-q^8)$. In other words, $$\deg_X \Phi_{J,T^2+T+1}(X,X)\leqslant \mathcal{N}_q.$$

\end{prop}

\begin{proof}
For $a_0$ not belong to cases in Lemma \ref{a0}, each pair $(a_0,a_1)$ satisfies system (\ref{5}) would correspond to a Drinfeld $\F_q[\theta]$-module $\psi_\theta$. The remaining thing is to find Drinfeld $A$-module $\phi_T(X)=TX+h_1X^q+h_2X^{q^2}+\Delta X^{q^3}$ such that $\phi_{T^2+T+1}=\psi_\theta$. Hence one needs to find $(h_1,h_2,\Delta)$ satisfies the system below:

\begin{equation}\label{7}
\begin{cases}
h_1+Th_1+h_1T^q=g_1\\
Th_2+h_2T^{q^2}+h_1^{q+1}+h_2=g_2\\
T\Delta+\Delta T^{q^3}+h_1h_2^q+h_2h_1^{q^2}+\Delta=g_3\\
h_1\Delta^q+\Delta h_1^{q^3}+h_2^{q^2+1}=g_4\\
h_1\Delta^{q^2}+\Delta h_2^{q^3}=g_5\\
\Delta^{q^3+1}=g_6

\end{cases}
\end{equation}

As we have $g_6\in \F_{q^2}^*$, and the facts that 
$$
\begin{array}{ll}
\Delta \textrm{ satisfies }& X^{q^3+1}-g_6=0\\
h_2 \textrm{ satisfies }&\Delta X^{q^3}+\Delta^{q^2} X-g_5=0\\
h_1 \textrm{ satisfies }&\Delta X^{q^3}+\Delta^q X+h_2^{q^2+1}-g_4=0

\end{array},
$$
for each given $\psi_\theta$, there are at most $(q^3+1)q^6$ many choices for $(h_1,h_2,\Delta)$. Therefore, when $a_0$ not belong to cases in Lemma \ref{a0}, each pair $(a_0,a_1)$ satisfying system (\ref{5}) would contribute at most $(q^3+1)q^6$ many Drinfeld $A$-modules with cyclic $(T^2+T+1)$-self isogeny. Thus there are at most $5q^5(6q^4-q^3)(q^3+1)q^6=5q^{14}(6q-1)(q^2-1)(q^3+1)$ many Drinfeld $A$-modules $\phi$ with  cyclic $(T^2+T+1)$-self isogeny.

When $a_0\in\F_{q^3}^*-\F_q^*$, there are at most $(q^2-1)(q^3-q)(q^2+q+1)q^2$ many triples $(g_6,a_0,a_1)$ that satisfy system $(\ref{iii})$. Each triple corresponds to a Drinfeld module $\psi_\theta$, hence there are at most $q^9(q^3+1)(q^2-1)^2(q^2+q+1)$ many Drinfeld $A$-modules with cyclic $(T^2+T+1)$-self isogeny.

When $a_0\in\F_{q^4}^*-\F_{q^2}^*$, there are at most $q^2(q^2-1)^2(q^4+q^2+q+1)$ many triples $(g_6,a_0,a_1)$ that satisfy system $(\ref{iii})$. Each triple corresponds to a Drinfeld module $\psi_\theta$, hence there are at most $q^8(q^3+1)(q^2-1)^2(q^4+q^2+q+1)$ many Drinfeld $A$-modules with cyclic $(T^2+T+1)$-self isogeny.

As a total, the number of Drinfeld $A$-modules $\phi$ with self $(T^2+T+1)$-cyclic monic isogeny is upper bounded by the number $(q^3+1)(q^2-1)(30q^{15}-4q^{14}+q^{12}+2q^{11}-q^{10}-2q^9-q^8)$.
\end{proof}
 \end{ex}
 
\section{Relation with Hilbert class polynomials}

\begin{defi}
Fix a basic $J$-invariant for rank-$r$ Drinfeld modules. For an order $\mathcal{O}$ of a degree-$r$ extension $K/F$, define the Hilbert class polynomial

$$H_{\mathcal{O},J}(X)=\prod_{\textrm{ isomorphism class }[\phi]\\ \textrm{ with  } {\rm End}(\phi)\cong\mathcal{O} }\left(X-J(\phi)\right)$$

\end{defi}

\begin{prop}\label{norm}

Let $\phi_1, \phi_2$ be Drinfeld $A$-modules over $\bar{F}$ of generic characteristic with complex multiplication by the same order ${\rm End}(\phi_1)\cong\mathcal{O}\cong {\rm End}(\phi_2)$ over $A$ in an imaginary degree-$r$ field extension $K/F$,  Then:
\begin{enumerate}
    \item For every finite place $\mathfrak{l}$, there is a canonical 
    \[
    \mathcal{O} \otimes_A A_\mathfrak{l}\text{-linear isomorphism } \quad 
    T_\mathfrak{l}(\phi_1) \;\cong\; T_\mathfrak{l}(\phi_2).
    \]

    \item For any $f \in \mathcal{O}$, let
    \[
    \mathrm{Nrd}_{\phi_i}(f) := \det\!\bigl(f \mid T_\mathfrak{l}(\phi_i)\bigr).
    \]
    Then $\mathrm{Nrd}_{\phi_1}(f) = \mathrm{Nrd}_{\phi_2}(f)$, i.e.\ the reduced norm of $f$ is independent of the choice of Drinfeld module with CM by $\mathcal{O}$.
\end{enumerate}
\end{prop}
\begin{proof}
\begin{enumerate}
\item The canonical isomorphism follows from the fact that  ${\rm End}(\phi_1)\cong\mathcal{O}\cong {\rm End}(\phi_2)$.

\item For a finite place $\fl$, the Tate module $T_\fl(\phi_i)$ is a free $A_\fl$-module of rank $r$. The endomorphisms of $\phi_i$ acts $A$-linearly on all torsion points, hence  $T_\fl(\phi_i)$ extends to a $\mathcal{O}\otimes_A A_\fl$-module of rank $1$. 

As $T_\fl(\phi_1)\cong T_\fl(\phi_2)$ as rank-$1$ $\mathcal{O}\otimes_A A_\fl$-modules, the isomorphism is canonical up to scaling by units in $\mathcal{O}\otimes_A A_\fl$. Thus the embedding $\mathcal{O}\hookrightarrow {\rm End}_{A_\fl}(T_\fl(\phi_i))$ depends only on the inclusion $\mathcal{O}\subset K$. Now for any $f\in\mathcal{O}$, the reduced norm $$ \mathrm{Nrd}_{\phi_1}(f)= \det\!\bigl(f \mid T_\mathfrak{l}(\phi_1)\bigr)={\rm Nm}_{K/F}(f)=\det\!\bigl(f \mid T_\mathfrak{l}(\phi_2)\bigr)=\mathrm{Nrd}_{\phi_2}(f).$$
Here ${\rm Nm}_{K/F}$ is the field norm of the extension $K/F$.

\end{enumerate}
\end{proof}

\begin{lem}
Let $\phi$ be a rank-$r$ Drinfeld module with ${\rm End}(\phi)=\mathcal{O}$ for an imaginary degree-$r$ order over $A$. Let $u\in {\rm End}(\phi)$, then we have $$r\cdot{\rm log}_q \deg_X(u)=\deg_T{\rm Nrd}_\phi(f).$$

\end{lem}

\begin{proof}
See \cite{Gek91}, Lemma 3.1.

\end{proof}
\begin{thm}\label{hilbert}
Fix a basic $J$-invariant for rank-$r$ Drinfeld modules, given a prime ideal $\mathfrak{a}\in A$ with monic generator $a$,

$$\Phi_{J,a}(X,X)=\prod_{\mathcal{O} \textrm{ imaginary $r$-degree order over $\F_q[T]$}}H_{\mathcal{O},J}(X)^{\gamma(\mathcal{O},a)},$$
where 
$$\begin{array}{ll}
\gamma(\mathcal{O},a)=&|\{ u\in \mathcal{O} \mid u \textrm{ primitive and }{\rm Nm}_{K/F}(u)=c\cdot a \textrm{ for some }c\in\F_q^*\}/\mathcal{O}^*|.
\end{array}
$$
Note that $u\in\mathcal{O}$ is primitive if $u\neq d\cdot \beta$ for some non-unit element $d\in A$ and $\beta \in \mathcal{O}$. Moreover, $\Phi_{J,a}(X,X)$ is invariant under ${\rm Gal}(\bar{F}/F)$-action, i.e. $\Phi_{J,a}(X,X)\in A[X]$.
\end{thm}

\begin{proof}
We begin with the following:

\begin{claim*}
$u\mathcal{O}$ is a cyclic A-sublattice of $\mathcal{O}$ isomorphic to $A/(a)$ if and only if $u$ is primitive and ${\rm Nm}_{K/F}(u)=c\cdot a \textrm{ for some }c\in\F_q^*$
\end{claim*}
\begin{proof}[proof of claim:] 
\begin{enumerate}
\item[$(\Rightarrow)$] 
From the fact that $|A/({\rm Nm}_{K/F}(u))|=|\mathcal{O}/u\mathcal{O}|=|A/(a)|$, we have ${\rm Nm}_{K/F}(u)=c\cdot a \textrm{ for some }c\in\F_q^*$. On the other hand, suppose $u$ is not primitive, then one can write $u=d\cdot \beta$ for some non-unit element $d\in A$ and $\beta \in \mathcal{O}$. Then we have $\mathcal{O}/u\mathcal{O}\subset \mathcal{O}/d\mathcal{O}$, which implies $u\mathcal{O}$ is not a cyclic A-sublattice of $\mathcal{O}$.
\item[$(\Leftarrow)$] Follows from the fundamental theorem for mulde over PID and the fact that ${\rm Nm}_{K/F}(u)=c\cdot a \textrm{ for some }c\in\F_q^*$ with $a\in A$ irreducible. 
\end{enumerate}

\end{proof}

From our definition of $\Phi_{J,a}(X,X)$, its roots are $J(\phi)$ counting with multiplicity that correspond to pairs $(\phi,u)$, up to isomorphic copy $(c^{-1}\phi c, c^{-1}uc)$ over $\mathbb{C}_\infty$. Hence
$$
\begin{array}{ll}
&\Phi_{J,a}(X,X)\\
\ \\
=&\prod_{\mathcal{O} \textrm{ imaginary $r$-degree order over $\F_q[T]$}}\prod_{[\phi] \textrm{ isomorphism class with }{\rm End}(\phi)\cong \mathcal{O}}\left(X-J(\phi)\right)^{\gamma(\mathcal{O},a)}\\
\ \\
			=&\prod_{\mathcal{O} \textrm{ imaginary $r$-degree order over $\F_q[T]$}}H_{\mathcal{O},J}(X)^{\gamma(\mathcal{O},a)}
\end{array}
$$

For $\Phi_{J,a}(X,X)$ is invariant under ${\rm Gal}(\bar{F}/F)$-action, it is enough to prove that for any $\sigma\in{\rm Gal}(\bar{F}/F)$, $$H_{\mathcal{O},J}(X)^\sigma=H_{\mathcal{O}^\sigma,J}(X).$$

Once the equality holds, the Galois conjugate $\mathcal{O}^\sigma$ is an imaginary degree-$r$ extension over $A$ with field of fraction $K^\sigma$. Hence from the definition of $\Phi_{J,a}(X,X)$, when $\mathcal{O}$ runs through all imaginary $r$-degree order, the Galois conjugates $\mathcal{O}^\sigma$ are included, hence $\Phi_{J,a}(X,X)^\sigma=\Phi_{J,a}(X,X)$.

The equality $H_{\mathcal{O},J}(X)^\sigma=H_{\mathcal{O}^\sigma,J}(X)$ follows from the property that  a Drinfeld module $\phi$ satisfies ${\rm End}(\phi)\cong \mathcal{O}$ if and only if ${\rm End}(\phi^\sigma)\cong \mathcal{O}^\sigma$.
\end{proof}

\section{Generalized Volcano Structure}
By considering $\fl$-isogeny graph for rank-$r$ Drinfeld modules with complex multiplication, our computation in previous sections actually correspond to the $J$-invariant of ``vertexes'' having self-horizontal edge. This relation motivates us to describe the shape of $\fl$-isogeny graph for CM Drinfeld modules, and its comparison with the $\ell$-isogeny graph for CM elliptic curves, which has a volcano structure.

In this section, we focus on general rank $r\geqslant2$, $K/F$ is a degree $r$ imaginary extension, and $\fl=(\ell)$ is a prime ideal of $A$. This section is to provide a characterization for imaginary degree-$r$ minimal orders whose corresponding Drinfeld modules have directed`` horizontal'' and ``vertical'' edges. 

\begin{defi}
Let $\fa$ be a nonzero integral $\mathcal{O}$-ideal. The {\bf{Fitting norm}} of $\fa$ over $A$ is the $A$-ideal
$${\rm{N}}_{\mathcal{O}/A}(\fa)={\rm Fitt}_A(\mathcal{O}/A)=\prod_{i=1}^m(d_i),$$
where $\mathcal{O}/A\cong \oplus_{i=1}^mA/(d_i),\textrm{ with }d_i\mid d_{i+1}$.

\end{defi}

\begin{rem}
Below are some properties for fitting norm:
\begin{enumerate}
\item $|\mathcal{O}/\fa|=|A/N_{\mathcal{O}/A}(\fa)|$.
\item Let $\ff:=\{x\in\cO_K\mid x\cO_K\subseteq \cO\}$ be the conductor of $\mathcal{O}$ in $K$. If $(\fa,\ff)=1$, then ${\rm{N}}_{\mathcal{O}/A}(\fa)={\rm{N}}_{K/F}(\fa)$. Here the norm on the right hand side is the field norm of $K/F$.
\item Let $\fa, \fb$ be $\mathcal{O}$-ideals with $(\fa,\fb)=1$, then ${\rm{N}}_{\mathcal{O}/A}(\fa\fb)={\rm{N}}_{\mathcal{O}/A}(\fa){\rm{N}}_{\mathcal{O}/A}(\fb).$ 

\end{enumerate}
\end{rem}

\begin{rem}
Let $\cO$ be an order of $K$ with conductor $\ff:=\{x\in\cO_K\mid x\cO_K\subset \cO\}$, unlike rank-$2$ case, we don't always have the equality $$\cO=A+\ff\cO_K.$$
\begin{cx}
$A=\F_7[T]$, $F=\F_7(T)$, and $K=F(\theta)$ where $\theta^4=T$. Consider the prime $\fp=(T-3)$ of $A$, there is a unique prime $\mathfrak{P}\mid \fp$ in $K$ with inertia degree $f(\mathfrak{P}\mid \fp)=4$, so $\cO_K/\mathfrak{P}\cong\F_{7^4}$.

Set $\ff=\fp\cO_K$, and consider the order $$\cO:=\{x\in \cO_K\mid x\mod \mathfrak{P} \textrm{ lies in }\F_{7^2} \textrm{ and }x\in \cO_{K,\fL}\ \forall \textrm{ prime } \fL\neq \mathfrak{P}\}.$$

Then conductor of $\cO$ is exactly $\ff$, but $A+\ff\cO_K\subsetneq \cO\subsetneq\cO_K$

\end{cx}
\end{rem}

\begin{defi}
An order $\cO$ of $K$ with conductor $\ff$ is called minimal if $\cO=A+\ff\cO_K$. 
\end{defi}

From now on, we consider rank-$r$ Drinfeld modules CM by a minimal order of $K$. We firstly characterize a higher rank version for horizontal $\fl^m$-isogenies.

\begin{thm}\label{hor}

Let $A=\F_q[T]$, $F=\F_q(T)$, $K/F$ be an imaginary degree-$r$ extension, and $\cO\subset \cO_K$ be a minimal order with conductor $\ff$. Let $\phi$ be a rank-$r$ Drinfeld module with ${\rm End}(\phi)\cong \cO$, then TFAE:

\begin{enumerate}
\item[(i)] There exists an $\fl$-cyclic isogeny $u:\phi\rightarrow \phi'$ with ${\rm ker}(u)\cong A/\fl$ and ${\rm End}(\phi')\cong \cO$.

\item[(ii)] $\fl \nmid \ff$ (so $\cO$ is locally maximal at $\fl$), and $\fl\cO_K=\prod_{i=1}^g\mathfrak{L}_i^{e_i}$ with $\mathfrak{L}_i\mid \fl$ prime of $\cO_K$, where some inertia degree ${f}(\mathfrak{L}_i\mid \fl)$ is equal to $1$.

\item[(iii)] There existis an invertible ideal $\fa$ of $\cO$ such that ${\rm N}_{K/F}(\fa)=\fl$.

\end{enumerate}
\end{thm}

\begin{proof}
We begin with the following lemma:

\begin{lem}\label{coprime}
Let $K/F$ be a finite separable extension, and $\cO\subset\cO_K$ be a minimal order of conductor $\ff\subset\cO_K$. For an integral $\cO$-ideal $\fa$, $$\fa \textrm{ is invertible in }\cO \textrm{ if and only if }\fa \textrm{ is coprime to } \ff$$
\end{lem}
\begin{proof}
Write $\cO=A+\ff\cO_K$, for each prime $\fp$ of $A$, if $\fp\nmid \ff$, then the localization $\cO_\fp=\cO_{K,\fp}$ is a product of discrete valuation rings. Thus any non-zero integral ideal $\fa$ in $\cO$ has $$(\fa_\fp,\ff_\fp)_\fp:=\fa_\fp+(\ff_\fp\cap \cO_\fp)=\fa_\fp+\cO_\fp=\cO_\fp.$$
On the other hand, for prime $\fp\mid \ff$, an integral and invertible $\cO$-ideal $\fa$ satisfies $\fa_\fp=\cO_\fp$ from the fact that $\cO_\fp$ is a local ring. Therefore, $$(\fa_\fp,\ff_\fp)_\fp=\fa_\fp+(\ff_\fp\cap \cO_\fp)=\cO_\fp.$$
Combining the two cases, we have $(\fa_\fp,\ff_\fp)_\fp=\cO_\fp$ for any prime ideal $\fp$ of $A$, this implies $(\fa,\fp)=\cO$. Thus $\fa$ is coprime to $\ff$. The converse statement is a well-known result.

\end{proof}

\begin{enumerate}
\item[(i)$\Rightarrow$(iii)]

Let $u:\phi\rightarrow \phi'$ be a $\fl$-cyclic isogeny with ${\rm End}(\phi')\cong \cO$, note that ${\rm ker}(u)\cong A/\fl$ as $A$-modules. Consider $$\fa:={\rm Ann}_\cO({\rm ker}(u))=\left\{\alpha\in \cO\mid \phi_\alpha|_{{\rm ker}(u)}=0\right\},$$
one can check that $\fa$ is an $\cO$-ideal. Moreover, as ${\rm End}(\phi)\cong\cO\cong {\rm End}(\phi')$ and the fact that $\phi'\cong\phi_\fa:=\phi/{\rm ker}(u)$,  one gets the right order $(\fa:\fa)=\{x\in K \mid x\fa\subset\fa\}$ is equal to ${\rm End}(\phi_\fa)\cong\cO$. Thus $\fa$ is an invertible ideal in $\cO$. From Lemma \ref{coprime}, $\cO/\fa\cong\phi[\fa]\cong{\rm ker}(u)\cong A/\fl$, and the fact that $|A/{\rm N}_{K/F}(\fa)|=|\cO/\fa|$ we have ${\rm N}_{K/F}(\fa)=\fl$

\item[(iii)$\Rightarrow$(ii)]
From Lemma \ref{coprime}, it remains to check the inertia degree, we write $\fl\cO_K=\prod_{i=1}^g\mathfrak{L}_i^{e_i}$. From the norm equality ${\rm N}_{K/F}(\fa)=\fl$, one can deduce that $\fa\cO_K$ is supported on some prime $\mathfrak{L}_j\mid \fl$ with inertia degree $f(\mathfrak{L}_j\mid \fl)=1$.

\item[(ii)$\Rightarrow$(i)]

Let $\mathfrak{L}$ be a prime of $cO_K$ above $\fl$ with inertia degree $f(\mathfrak{L}\mid \fl)=1$, we consider the $\cO$-ideal $$\fa:=\mathfrak{L}\cap \cO.$$
Because  $f(\mathfrak{L}\mid \fl)=1$ and $\fl\nmid\ff$, we have $\fa$ is invertible in $\cO$ and ${\rm N}_{K/F}(\fa)=\fl$. Thus $\fa$ induces an isogeny $\iota_\fa:\phi\rightarrow\phi_\fa$ with ${\rm ker}(u)\cong\cO/\fa\cong A/{\rm N}_{K/F}(\fa)=A/\fl$, and  ${\rm End}(\phi_\fa)=(\fa:\fa)=\cO.$

\end{enumerate}
\end{proof}

In a more general situation, the types of cyclic and non-cyclic $\fl$-power horizontal isogenies from $\phi$ are classified in the following theorem:

\begin{thm}\label{lmhor}
Let $\cO\subset\cO_K$ be a minimal order with conductor $\ff$. Let $\fl\cO_K=\prod_{i=1}^g\mathfrak{L}_i^{e_i}$ with inertia degree $f_i:=f(\mathfrak{L}_i\mid \fl)$. For each $m\geqslant 1$, the $\fl^m$-isogenies $\phi\rightarrow \phi'$ that satisfy ${\rm End}(\phi')\cong \cO$ correspond to invertible $\cO$-ideals $\fb$ with
$$\mathfrak{b}\cO_K=\prod_{i=1}^g\mathfrak{L}_i^{m_i}\ ,m_i\in\mathbb{Z}_{\geqslant 0},$$
with ${\rm N}_{K/F}(\mathfrak{b})=\fl^m$, i.e. $\sum_{i=1}^g f_im_i=m$. For such an ideal $\mathfrak{b}$ we have the induced isogeny $\iota_{\mathfrak{b}}:\phi\rightarrow \phi_{\mathfrak{b}}$ satisfies $${\rm ker}(\iota_\fb)\cong\oplus_{i=1}^g\cO/\mathfrak{L}_i^{m_i}\cong \oplus_{i=1}^g A/\fl^{f_im_i}.$$
In particular, 
\begin{enumerate}
\item[$(i)$]$\iota_\fb$ is $\fl^m$-cyclic if and only if there is exactly one index $i$ having $m_i>0$, which forces $f_i\mid m$.
\item[$(ii)$]$\iota_\fb$ is a non-cyclic $\fl^m$-isogeny if and only if there are at least two $m_i>0$ from the equation $\sum_{i=1}^g f_im_i=m$.
\item[$(iii)$] $\iota_\fb$ is a self isogeny, i.e. $\phi_\fb\cong\phi$ , if and only if $\fb$ is principal in $\cO$.
\end{enumerate}
\end{thm}
\begin{proof}
First of all, a rank-$r$ Drinfeld $A$-module CM by $\cO$ can also be viewed as a rank-$1$ Drinfeld $\cO$-module. Thus a horizontal $\fl^m$-isogeny $u:\phi\rightarrow \phi'$ with ${\rm End}(\phi')\cong \cO$ corresponds to an invertible $\cO$-ideal $\fb$ with $(\fb,\ff)=\cO$, and its induced $\fl^m$-isogeny $\iota_\fb:\phi\rightarrow \phi_\fb$, where $\phi_\fb\cong\phi'$ and $\cO\cong {\rm End}(\phi_\fb)=(\fb:\fb)$. As $\deg(\iota_\fb)=|\cO/\fb|=|A/{\rm N}_{K/F}(\fb)|=|A/\fl^m|$, we have the ideal $\fb\cO_K$ is a product of prime ideals $\mathfrak{L}_i$ of $\cO_K$ stand above $\fl$, i.e. $\mathfrak{b}\cO_K=\prod_{i=1}^g\mathfrak{L}_i^{m_i}\ ,m_i\in\mathbb{Z}_{\geqslant 0}.$ Now ${\rm N}_{K/F}(\fb)=\fl^m$ implies $\sum_{i=1}^g f_im_i=m$, and the isomorphisms on ${\rm ker}(\iota_\fb)$ follows from the fact that $\cO/\mathfrak{L}_i^{m_i}\cong A/{\rm N}_{K/F}(\mathfrak{L}_i)^{m_i}\cong A/\fl^{f_im_i}$.

Next, we consider a special case where $\iota_\fb$ is a $\fl^m$-cyclic isogeny. This happens if and only if ${\rm ker}(\iota_\fb)\cong A/\fl^m$, which means the factorization $\mathfrak{b}\cO_K=\prod_{i=1}^g\mathfrak{L}_i^{m_i}$ has only one non-trivial piece $\fb\cO_K=\mathfrak{L}_i^{m_i}$, and $f_im_i=m$. 

For the case $\iota_\fb$ is a non-cyclic $\fl^m$-isogeny. It happens if and only if ${\rm ker}(\iota_\fb)$ as an $A$-module has at least two pieces. In other words, the factorization $\mathfrak{b}\cO_K=\prod_{i=1}^g\mathfrak{L}_i^{m_i}$ has at least two $m_i>0$, and $\sum_{i=1}^g f_im_i=m$.

Finally, for the case $\iota_\fb$ is a self isogeny. Write $\phi=\phi_I$ for a fraction $\cO$-ideal $I$. Then this happens if and only if $\phi_{\fb^{-1}I}\cong\phi_I$, which means $\fb^{-1}I=uI$ for some $u\in K^*$. By multiplying $I^{-1}$ on both sides, one gets $\fb^{-1}=u\cO$, so $\fb$ is principal in $\cO$.
\end{proof}

Theorem \ref{hor} gives an equivalence conditions on horizontal $\fl$-cyclic isogeny admitted from a Drinfeld module CM by a minimal order $\cO$. To consider ``vertical'' isogenies, we need to define a suitable level for ${\rm End}(\phi)$ provided a rank-r Drinfeld module $\phi$ with CM by $K$.

\begin{prop}\label{cond}

Let $\fl\cO_K=\prod_{i=1}^g\fL_i^{e_i}$. By localization at $\fl$, one has $\cO_{K,\fl}\cong\prod_{\fL\mid \fl}\cO_{K,\fL}$. Let $t=(t_L)_L\in\mathbb{Z}_{\geqslant 0}^g$ and define $R_{\fl,t}=A_\fl+\prod_{\fL\mid \fl}L^{t_L}\cO_{K,\fL}$, then we have 

\begin{enumerate}
\item[(i)] $R_{\fl,t}\subset\prod_{\fL\mid \fl}\cO_{K,\fL}$ is a local ring with maximal ideal $\mathfrak{m}_t=\fl A_\fl+\prod_{\fL\mid \fl}L^{t_L}\cO_{K,\fL}$

\item[(ii)] The right order $(\mathfrak{m}_t:\mathfrak{m}_t)=\{x\in \prod_{\fL\mid \fl}K_\fL\mid \mathfrak{m}_tx\subset\mathfrak{m}_t\}$ is equal to $A_\fl+\prod_{\fL\mid \fl}L^{{\rm max}(t_L-1,0)}\cO_{K,\fL}$
\end{enumerate}

\end{prop}
\begin{proof}
\begin{enumerate}
\item[(i)] One can easily observe that $\mathfrak{m}_t$ is a maximal ideal of $R_{\fl,t}$ because the isomorphism of quotient rings $f: R_{\fl,t}/\mathfrak{m}_t\cong A_\fl/\fl A_{\fl}$. To prove $R_{\fl,t}$ is local, it is enough to prove that $x\in R_{\fl,t}-\mathfrak{m}_t$ is a unit. We write $x\in R_{\fl,t}-\mathfrak{m}_t$ as $x=a+u$, where $a\in A_\fl-\fl A_\fl$ and $u\in \prod_{\fL\mid \fl}L^{t_L}\cO_{K,\fL}$. Thus $a\in A_\fl^*$ and so $x=a(1+a^{-1}u)$. One then check that $1+a^{-1}u$ is a unit in $R_{\fl,t}$ because its image under $f$ is $1\in A_\fl/\fl A_\fl$. 
\item[(ii)]

Take $x=a+u\in A_\fl+\prod_{\fL\mid \fl}L^{{\rm max}(t_L-1,0)}\cO_{K,\fL}$, we firstly prove that $\mathfrak{m}_tx\subset\mathfrak{m}_t$. The product $x\mathfrak{m}_t=a\fl A_\fl+\left[u\fl A_\fl+(a+u)\prod_{\fL\mid \fl}L^{t_\fL}\cO_{K,\fL}\right] \subset \fl A_\fl+\prod_{\fL\mid \fl}L^{t_\fL}\cO_{K,\fL}=\mathfrak{m}_t$.

Conversely, let $x\in (\mathfrak{m}_t:\mathfrak{m}_t)$, one can deduce that as $A_\fl$-modules, the multiplication-by-$x$ map induces $A_\fl$ linear map on $\mathfrak{m}_t$ whose characteristic polynomial is defined over $A_\fl$. Hence $x\in\cO_{K,\fl}$.

Write $x=(x_\fL)_{\fL\mid\fl}\in \cO_{K,\fl}$, consider the diagonal reduction $\pi:R_{\fl,t}\rightarrow \prod_{\fL\mid \fl}\cO_{K,\fL}/\fL$, it has kernel ${\rm ker}(\pi)=\mathfrak{m}_t$ and $R_{\fl,t}/\mathfrak{m}_t\cong A_\fl/\fl A_\fl$. As $x\in(\mathfrak{m}_t:\mathfrak{m}_t)$, the multiplication-by-$x$ map induces a well-define $A_\fl$-linear map $\bar{\lambda}_x:R_{\fl,t}/\mathfrak{m}_t\rightarrow R_{\fl,t}/\mathfrak{m}_t$. Let $\rho:\cO_{K,\fl}\rightarrow\prod_{\fL\mid \fl}\cO_{K,\fL}/\fL$ and set $\rho(x)=(\bar{x}_\fL)_\fL$, then for any $a\in A_\fl$, we have
$$\bar{\lambda}_x(\pi(a))=\pi(xa)=\rho(x)\pi(a)=(\bar{x}_\fL\bar{a})_\fL\in A_\fl/\fl A_\fl,$$
thus all $\fL$-components of $(\bar{x}_\fL)_\fL$ are equal.  Set $n_\fL:={\rm max}\{t_\fL-1,0\}$ and $N:={\rm max}_{\fL\mid \fl}\{n_\fL\}$, consider the map $\iota:A_\fl/\fl^N\rightarrow \prod_{\fL\mid\fl}\cO_{K,\fL}/\fL^{n_\fL}$ induced from the diagonal embedding of $A_\fl$ into $\prod_{\fL\mid \fl}\cO_{K,\fL}$. Because $\bar{x}_\fL$ are equal for all $\fL\mid \fl$, one can apply Hensel's lifting lemma to find an element $a\in A_\fl$ such that $$x_\fL\equiv a \mod \fL^{n_\fL} \textrm{ for all }\fL\mid \fl.$$ Thus we have $x-a\in\prod_{\fL\mid \fl}L^{{\rm max}(t_L-1,0)}\cO_{K,\fL}$, the proof is now complete.

\end{enumerate}
\end{proof}

\begin{defi}\label{lev}
Let $\fl$ be a prime ideal of $A$ with $\fl\cO_K=\prod_{i=1}^g\fL_i^{e_i}$. 
\begin{enumerate}
\item[(i)] Let $\phi$ be a rank-$r$ Drinfeld module CM by a minimal order $\cO$ of conductor $\ff$, consider the prime factorization of $\ff$ as $$\ff\cO_K=\ff_0\cdot \prod_{i=1}^g\fL_i^{t_i},$$
where $\ff_0$ is coprime to $\fl$. Define the level of $\phi$ to be the $g$-tuple $(t_{\fL_1},\cdots,t_{\fL_g})\in\mathbb{Z}_{\geqslant 0}^g$, and we say $\phi$ has uniform level $t$ if $t_{\fL_i}=t$  for $1\leqslant i\leqslant g$.

\item[(ii)] Define the $\fl$-cyclic isogeny graph $\mathcal{G}_{\fl,K}$ as a directed multigraph as follows: \\
The vertex set $\mathcal{V}$ of $\mathcal{G}_{\fl,K}$ is the set of $\mathbb{C}_\infty$-isomorphism classes of rank-$r$ Drinfeld modules $[\phi]$ CM by $K$. A directed edge from $[\phi_1]$ to $[\phi_2]$ is a cyclic $\fl$-isogeny $u:\phi_1\rightarrow \phi_2$ up to isomorphic copy $c^{-1}uc$ over $\mathbb{C}_\infty$.

\item[(iii)] In $\mathcal{G}_{\fl, K}$ we define its minimal, and uniform component $\mathcal{G}_{\fl,K}^{\rm min, unif}$ to be the subgraph induced by vertices $[\phi]$ CM by minimal orders in $K$ having uniform level. 

\item[(iv)] In $\mathcal{G}_{\fl,K}^{\rm min, unif}$, an edge $u:\phi_1\rightarrow \phi_2$ is called horizontal if the levels $t_1$ of$[\phi_1]$ and $t_2$ of $[\phi_2]$ are the same. It is called ascending if $t_2>t_1$, and descending if $t_2<t_1$.
\end{enumerate}
\end{defi}

The following theorem shows a volcano-like structure for the graph $\mathcal{G}_{\fl,K}^{\rm min, unif}$:

\begin{thm}\label{voc}
\begin{enumerate}
\item[(i)]
In $\mathcal{G}_{\fl,K}^{\rm min, unif}$, vertex  $[\phi]$ with level $t\geqslant 1$ has only one ascending $\fl$-cyclic isogeny to a vertex $[\phi']$ of level $t-1$ and preserve the $\fl$-comprime factor $\ff_0$. Besides, there is no horizontal or descending edge.

\item[(ii)] For vertex $[\phi]\in \mathcal{G}_{\fl,K}^{\rm min, unif}$ with level $t=0$, the horizontal edge are $\fl$-cyclic isogenies $u:\phi\rightarrow \phi'$ with ${\rm End}(\phi')\cong \cO$, the number of horizontal edges is the number of primes $\fL$ of $\cO_K$ above $\fl$ with inertia degree $f(\fL\mid \fl)=1$ . Note that $[\phi]$ has no descending edges.

\end{enumerate}
\end{thm}

\begin{proof}
For proof of (i), let $\phi$ be a Drinfeld module CM by a minimal order $\cO$ of conductor $\ff$, where $\ff\cO_K=\ff_0\cdot \prod_{i=1}^g\fL_i^{t}$ for some $t\in\mathbb{Z}_{\geqslant 1}$. As a $\fl$-cyclic isogeny outwarded from $\phi$ correspond to an  $\cO$-ideal $\fa$ with $\cO/\fa\cong A/\fl$. Thus it is enough to prove that $\fa=\fl A+\ff\cO_K$ is the only $\cO$-ideal such that $\cO/\fa\cong A/\fl$. From Proposition \ref{cond}(i), we have the property that for any element $z\in\cO-\fa$, it generates the $\cO$-ideal $z\cO$ which is coprime to $\fl$. The property that $\cO/\fa\cong A/\fl$ is clear. Now $\fa$ induces an $\fl$-cyclic isogeny $\iota_\fa:\phi\rightarrow \phi_\fa$ with ${\rm End}(\phi_\fa)=(\fa:\fa)$. We can deduce from Proposition \ref{cond}(ii) that $(\fa:\fa)=\fl A+\ff_0\prod_{i=1}^g\fL_i^{t-1}\cO_K$.

For (ii),  it follows from Theorem \ref{lmhor} (i) with the case $m=1$.  Let $\phi$ be a Drinfeld module CM by a minimal order $\cO$ of conductor $\ff$, where $\ff\cO_K=\ff_0\cO_K$. Then $[\phi]$ has no descending edges because for every $\cO$-ideal $\fa$ with $\cO/\fa\cong A/\fl$, we have that  $\fa$ is coprime to $\ff_0$.

\end{proof}

For non-uniform component of $\mathcal{G}_{\fl,K}^{\rm min}$, we show that every vertex $[\phi]$ has only one ascending edge,  no horizontal or vertical edges. 

\begin{prop}\label{nonunif}
Let $[\phi]$ be a non-uniform vertex with level $(t_{\fL_1},\cdots,t_{\fL_g})$, then it has only one ascending $\fl$-cyclic isogeny to a vertex $[\phi']$ of level $({\rm max}(t_{\fL_i}-1,0))_{1\leqslant i\leqslant g}$ and preserve the $\fl$-comprime factor $\ff_0$.
\end{prop}
\begin{proof}
Similar proof as Theorem \ref{voc}(i), while starting with a Drinfeld module $\phi$ CM by a minimal order $\cO$ of conductor $\ff$ with $\ff\cO_K=\ff_0\prod_{i=1}^g\fL_i^{{\rm max}(t_{\fL_i}-1,0)}\cO_K$.

\end{proof}

Similar to $\ell$-isogeny graph for CM elliptic curves, we compute the number of level-$t$ vertices $[\phi]$ in $\mathcal{G}_{\fl, K}^{\rm min, unif}$ having ascending edge to the same vertex of level (t-1). 

\begin{prop}\label{deg}
For $t\geqslant 1$, suppose there is a level-$t$ vertices in $\mathcal{G}_{\fl, K}^{\rm min, unif}$ having ascending edge to the a vertex $[\phi']$ of level $t-1$, where ${\rm End}(\phi')\cong \cO_{\fl,t-1}:=A+\ff_0\prod_{i=1}^g\fL_i^{t-1}\cO_K$. Then the number of level-$t$ vertices having ascending edge toward $[\phi']$ is equal to cardinality of ${\rm ker}(\theta_t)$, where $\theta_t$ is the map $$\theta_t:{\rm Pic}(\cO_{\fl,t})\rightarrow {\rm Pic}(\cO_{\fl,t-1})\ \textrm{ via } [\fa]\mapsto [\fa\cO_{\fl,t-1}],$$
with $\cO_{\fl,t}:=A+\ff_0\prod_{i=1}^g\fL_i^{t}\cO_K$. Moreover, we have $|{\rm ker}(\theta_t)|=|A/\fl|^{r-1}$.
\end{prop}

\begin{proof}

Let $[\phi']$ be a vertex in $\mathcal{G}_{\fl, K}^{\rm min, unif}$ of level $t-1$, and ${\rm End}(\phi')\cong \cO_{\fl,t-1}:=A+\ff_0\prod_{i=1}^g\fL_i^{t-1}\cO_K$. The level-$t$ vertices $[\phi]$ having ascending edge to $[\phi']$ must have ${\rm End}(\phi)\cong \cO_{\fl,t}$ from Theorem \ref{voc} (i). From CM uniformization (ref. section 2.2) it is well known that there is a 1-1 correspondence between the isomorphism classes of Drinfeld modules CM by $\cO_{\fl,t}$ and elements in the Picard group ${\rm Pic}(\cO_{\fl,t})$.  Indeed, one has $$[\fa]\in {\rm Pic}(\cO_{\fl,t})\leftrightarrow \cO_{\fl,t}\textrm{-lattice }\Lambda=\fa^{-1}\leftrightarrow \phi_\Lambda.$$

One can see that for a level-$t$ class $[\phi_\Lambda]$ that corresponds to $[\fa]\in{\rm Pic}(\cO_{\fl,t})$, it has the ascending edge $\iota_{\mathfrak{m}_t}:\phi_{\fa}\rightarrow \phi_{\mathfrak{m}_t^{-1}\Lambda}$ induced from the maximal ideal $\mathfrak{m}_t=\fl A+\ff_0\prod_{i=1}^g\fL^t\cO_K$ of $\cO_{\fl,t}$. Thus we have
$$[\fa]\in {\rm Pic}(\cO_{\fl,t})\leftrightarrow \cO_{\fl,t}\textrm{-lattice }\Lambda=\fa^{-1}\leftrightarrow \phi_\Lambda\xrightarrow{\iota_{\mathfrak{m}_t}}\phi_{\mathfrak{m}_t^{-1}\Lambda}\leftrightarrow \cO_{\fl,t-1}\textrm{-lattice }\mathfrak{m}_t^{-1}\fa^{-1}\leftrightarrow [\mathfrak{m}_t\fa]\in {\rm Pic}(\cO_{l\fl,t-1}) $$
Therefore, two classes $[\fa_1], [\fa_2]\in {\rm Pic}(\cO_{\fl,t})$ having ascending $\fl$-cyclic isogenies toward the same vertex of level $t-1$ if and only if $\mathfrak{m}_t^{-1}\fa_1^{-1}\cO_{\fl,t-1}=\mathfrak{m}_t^{-1}\fa_2^{-1}\cO_{\fl,t-1}$, which is equivalent to say $\fa_1\cO_{\fl,t-1}=\fa_2\cO_{\fl,t-1}$, i.e. $[\fa_1\fa_2^{-1}]\in {\rm ker}(\theta_t)$.

Now let's compute the cardinality of $\theta_t$, we claim that 
$${\rm ker}(\theta_t)\cong\frac{\prod_{i=1}^g(1+\fL^{t-1})/(1+\fL^t)}{1+\fl^{t-1}A/1+\fl^t A},$$
where the denominator embeds to numerator via the diagonal embedding $\iota: A_\fl^*\hookrightarrow \prod_{i=1}^g\cO_{K,\fL_i}^*$, so $|{\rm ker}(\theta_t)|=|A/\fl|^{r-1}$ by direct computation. Define $\mathcal{F}_m=\prod_{i=1}^g\fL^m\cO_K$ for $m\in\mathbb{Z}_{\geqslant 0}$, and the group homomorphism $$\Phi:1+\mathcal{F}_{t-1}\rightarrow {\rm Pic}(\cO_{\fl,t}) \textrm{ via } x\mapsto [x^{-1}\cO_{t-1}\cap\cO_{t}].$$
It's clear that $\phi(x)\in{\rm ker}(\theta_t)$ and $\Phi$ factors through $1+\mathcal{F}_t$. Furthermore, for any $a\in1+\fl^{t-1}A$ and $x\in 1+\mathcal{F}_{t-1}$, 
$$\Phi(ax)=[(ax)^{-1}\cO_{\fl,t-1}\cap\cO_{\fl,t}]=[x^{-1}(a^{-1}\cO_{\fl,t-1})\cap\cO_{\fl,t}]=[x^{-1}\cO_{\fl,t-1}\cap\cO_{\fl,t}][a^{-1}\cO_{\fl,t-1}\cap\cO_{\fl,t}]=\Phi(x).$$
Hence $\Phi$ descends to $\tilde{\Phi}:\frac{1+\mathcal{F}_{t-1}/1+\mathcal{F}_t}{1+\fl^{t-1}A/1+\fl^t A}\rightarrow {\rm ker}(\theta_t).$ Combining with the Chinese Remainder Theorem that implies $1+\mathcal{F}_{t-1}/1+\mathcal{F}_t\cong\prod_{i=1}^g(1+\fL^{t-1})/(1+\fL^t)$, it is enough to show $\tilde{\Phi}$ is bijective.

To show $\tilde{\Phi}$ is surjective, let $[\fa]\in {\rm ker}(\theta_t)$, then we have $\fa\cO_{\fl,t-1}=x\cO_{\fl,t-1}$ for some $x\in K^*$. As $\fa$ is an invertible $\cO_{\fl,t}-ideal$, we can deduce from the above equality that the valuation $v_\fL(x^{-1})=0$ for any prime $\fL$ of $\cO_K$ stands above $\fl$. By multiplying $x^{-1}$ with a suitable element in $\cO_K$, we may assume that $$x^{-1}\in \cO_{\fl,t-1}, \textrm{ and }v_\fL(x^{-1})=0\ \forall\fL\mid \fl.$$
Reducing modulo $\mathcal{F}_{t-1}$, one gets $\bar{x^{-1}}\in (\cO_{\fl,t-1}/\mathcal{F}_{t-1})^*\cong (A/A\cap \mathcal{F}_{t-1})^*$, thus one can find $a\in A^*$ such that $u:=x^{-1}a^{-1}\in 1+\mathcal{F}_{t-1}$, then $\tilde{\Phi}(u)=[x\cO_{\fl,t-1}\cap\cO_{\fl,t}]=[\fa]$. For injectivity of $\tilde{\Phi}$, if $\tilde{\Phi}(u_1)=\tilde{\Phi}(u_2)$, then$[(u_1/u_2)^{-1}\cO_{\fl,t-1}\cap\cO_{\fl,t}]=[\cO_{\fl,t}]$, which implies $u_1/u_2\in (1+\mathcal{F}_t)(1+\fl^{t-1}A)$. Hence $u_1$ and $u_2$ represent the same class in $\frac{1+\mathcal{F}_{t-1}/1+\mathcal{F}_t}{1+\fl^{t-1}A/1+\fl^t A}$.

\end{proof}

As a conclusion, we formulate the isogeny volcano structure for CM Drinfeld modules.

\begin{defi}
Fix a prime ideal $\fl$ of $A=\F_q[T]$, and an integer $r\geqslant 2$. A {\bf generalized $\fl$-cyclic volcano $V$} is a directed graph whose vertices partitioned in terms of levels $V_0, V_1, \cdots, V_d,\cdots$ such that the following hold:

\begin{enumerate}
\item[(i)] The subgraph on $V_0$ (the crater) is a regular graph of degree at most $r$. 
\item[(ii)] For $t>0$, each vertex in $V_t$ has exactly one upward edge toward a neighbor in $V_{t-1}$
\item[(iii)] If the partition of vertices stop at $V_d$ (the floor), then we say $d$ is the depth of $V$. Any vertex in $V_i$ with $1\leqslant i< d$ has out-degree $1$ and in-degree $|A/\fl|^{r-1}$;  vertices in $V_d$ has out-degree $1$ and in-degree $0$.
\item[(iv)] If $V$ has infinite depth, then any vertex in $V_t$ with $t>0$ must have one of followings:
\begin{itemize}
\item[(a)] out-degree $1$ and in-degree $|A/\fl|^{r-1}$
\item[(b)] out-degree $1$ and in-degree $0$.
\end{itemize}

\end{enumerate}

\end{defi}

\begin{rem}
\begin{enumerate}
\item[1.] From Theorem \ref{voc} and Proposition \ref{deg}, we know the minimal and uniform component $\mathcal{G}_{\fl,K}^{\rm min,unif}$ (see Definition \ref{lev}(iii)) is a generalized $\fl$-cyclic volcano.

\item[2.] As a comparison between $\mathcal{G}_{\fl,K}^{\rm min,unif}$ and the $\ell$-isogeny graph $\mathcal{G}_{k,\ell,\ff_0}$ with $\ff_0^2\Delta_K<-4$ for CM elliptic curves studied by Clark (\cite{C22}, section 4.2 ), we make the following chart:
\center
\begin{tikzpicture}
\clip node (m) [matrix,matrix of nodes,
fill=black!20,inner sep=0pt,
nodes in empty cells,
nodes={minimum height=1cm,minimum width=2.6cm,anchor=center,outer sep=0,font=\sffamily},
row 1/.style={nodes={fill=black,text=white}},
column 1/.style={nodes={fill=gray,text=white,align=right,text width=3cm,text depth=0.5ex}},
column 2/.style={text width=4cm,align=center,every even row/.style={nodes={fill=white}}},
column 3/.style={text width=7cm,align=center,every even row/.style={nodes={fill=white}},},
row 1 column 1/.style={nodes={fill=gray}},
prefix after command={[rounded corners=4mm] (m.north east) rectangle (m.south west)}
] {
                &       $ \mathcal{G}_{k,\ell,\ff_0}  $           &$\mathcal{G}_{\fl,K}^{\rm min,unif} $ \\
Depth 		& $\infty$					&	$\infty$ \\
Vertex at crater     & degree$=1+\left(\frac{\Delta_K}{\ell}\right)$ & out-degree$=g_1$; in-degree$=g_1+|A/\fl|^{r-1}$\\
Vertex below crater     & degree$=\ell+1$                              &
												 out-degree $1$ and in-degree $|A/\fl|^{r-1}$												\\
 Notes   			&	If $\ff_0\Delta_K\in\{-3,-4\}$, then 	$ \mathcal{G}_{k,\ell,\ff_0}  $ is not a $\ell$-volcano				&       There are non-uniform components and non-minimal orders in full $\fl$-cyclic isogeny graph                      \\
};
\end{tikzpicture}
Here $g_1:=\#\{\textrm{ prime } \fL| \fl \textrm{ in }\cO_K \textrm{ with } f(\fL|\fl)=1\}$.\\

\item[3.] When $r=2$, we know that every order of a quadratic imaginary extension $K/F$ is minimal, and the level of an order in $\cO_K$ is always uniform because the split types for a prime $\fl$ of $A$ in $\cO_K$ is simple: split, inert, or ramified. Therefore, we have $$\mathcal{G}_{\fl,K}^{\rm min,unif}=\mathcal{G}_{\fl,K}.$$
Thus the $\fl$-cyclic isogeny graph for rank-$2$ Drinfeld modules with CM by $K$ has a generalized volcano structure, and it is compatible with the volcano structure defined by Caranay \cite{C18} Definition 2.4.11.

\end{enumerate}

\end{rem}

\begin{ex}\label{r3case}
In this example, we show a brief picture on a connected component of $\mathcal{G}_{\fl,K}^{\rm min,unif} $  for the case $q=5, r=3$, and $\fl$ a prime of $A$ with split type $(1,2)$ in $\cO_K$:

\begin{enumerate}
\item[\bf cycle case:] 
Consider $K=F(\sqrt[3]{T^3+T+1})$, $\fl=(T)$, and $\cO=\cO_K$. Then $K/F$ is imaginary and $\fl\cO_K=\fL_1\fL_2$ with $f(\fL_1|\fl)=1$ and $f(\fL_2|\fl)=2$. Note that $\fL_1=(T,\sqrt[3]{T^3+T+1})$. The class number $|{\rm Pic}(\cO_K)|=6$ because $K$ is a genus-$1$ function field and there are $6$ many $\F_5$-rational point of the affince curve $\mathcal{C}: Y^3=X^3+X+1$ Now we compute the order of $[\fL_1]\in{\rm Pic}(\cO_K)$. One can compute directly to show that$[\fL_1^2]$ and $[\fL_1^3]$ are nontrivial, so  the order of $[\fL_1]\in{\rm Pic}(\cO_K)$ is $6$. Hence a connected component of  $\mathcal{G}_{\fl,K}^{\rm min,unif} $ should be of the shape below.

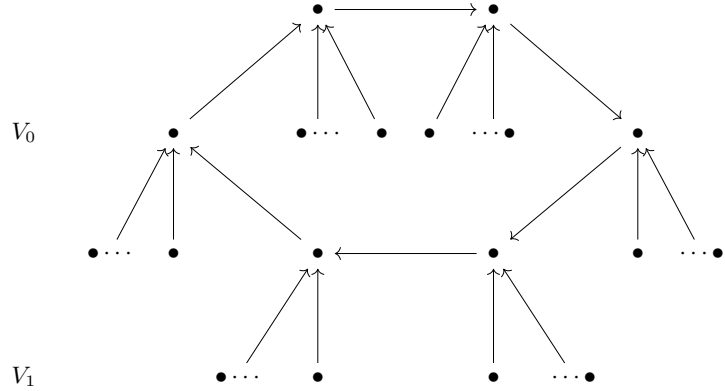
\[\begin{tikzcd}[column sep=tiny]
	&&&&&&&& \bullet &&& \bullet \\
	\\
	&&& {V_0} &&& \bullet && {\bullet\cdots} & \bullet & \bullet & {\cdots\bullet} && \bullet \\
	\\
	&&&&& {\bullet\cdots} & \bullet && \bullet &&& \bullet && \bullet & {\cdots\bullet} \\
	\\
	{} &&& {V_1} &&&& {\bullet\cdots} & \bullet &&& \bullet & {\cdots\bullet}
	\arrow[from=1-9, to=1-12]
	\arrow[from=1-12, to=3-14]
	\arrow[from=3-7, to=1-9]
	\arrow[from=3-9, to=1-9]
	\arrow[from=3-10, to=1-9]
	\arrow[from=3-11, to=1-12]
	\arrow[from=3-12, to=1-12]
	\arrow[from=3-14, to=5-12]
	\arrow[from=5-6, to=3-7]
	\arrow[from=5-7, to=3-7]
	\arrow[from=5-9, to=3-7]
	\arrow[from=5-12, to=5-9]
	\arrow[from=5-14, to=3-14]
	\arrow[from=5-15, to=3-14]
	\arrow[from=7-8, to=5-9]
	\arrow[from=7-9, to=5-9]
	\arrow[from=7-12, to=5-12]
	\arrow[from=7-13, to=5-12]
\end{tikzcd}\]
Each vertex in $V_0$ would have inward edges from $|\F_5[T]/(T)|^{3-1}=25$ vertices in $V_1$, this trend is followed for layer $V_d$ with $d\geqslant 2$.

\item[\bf self-loop case:]
Consider $K=F(\sqrt[3]{T+1})$, $\fl=(T)$, and $\cO=\cO_K$. Again$K/F$ is imaginary and $\fl\cO_K=\fL_1\fL_2$ with $f(\fL_1|\fl)=1$ and $f(\fL_2|\fl)=2$.  Moreover, the class number $|{\rm Pic}(\cO_K)|=1$ because $K$ is a genus $0$ function field. Hence a connected component of  $\mathcal{G}_{\fl,K}^{\rm min,unif} $ should be of the shape below.

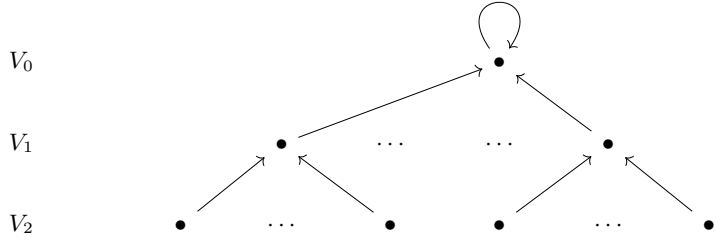
\[\begin{tikzcd}
	{V_0} &&&&& \bullet \\
	{V_1} &&& \bullet & \cdots & \cdots & \bullet \\
	{V_2} && \bullet & \cdots & \bullet & \bullet & \cdots & \bullet
	\arrow[from=1-6, to=1-6, loop, in=55, out=125, distance=10mm]
	\arrow[from=2-4, to=1-6]
	\arrow[from=2-7, to=1-6]
	\arrow[from=3-3, to=2-4]
	\arrow[from=3-5, to=2-4]
	\arrow[from=3-6, to=2-7]
	\arrow[from=3-8, to=2-7]
\end{tikzcd}\]

Here in the layer $V_1$ there are $|\F_5[T]/(T)|^{3-1}=25$ vertices, and each vertex in $V_1$ would have inward edges from $25$ vertices in $V_2$, this trend is followed for layer $V_d$ with $d\geqslant 3$

\end{enumerate}

\end{ex}

\bibliographystyle{alpha}
\bibliography{ModularPolyDM.bib}

@book {G96,
    AUTHOR = {Goss, David},
     TITLE = {Basic structures of function field arithmetic},
    SERIES = {Ergebnisse der Mathematik und ihrer Grenzgebiete (3) [Results
              in Mathematics and Related Areas (3)]},
    VOLUME = {35},
 PUBLISHER = {Springer-Verlag, Berlin},
      YEAR = {1996},
     PAGES = {xiv+422},
      ISBN = {3-540-61087-1},
   MRCLASS = {11G09 (11L05 11R58)},
  MRNUMBER = {1423131},
MRREVIEWER = {Jeremy T. Teitelbaum},
       DOI = {10.1007/978-3-642-61480-4},
       URL = {https://doi.org/10.1007/978-3-642-61480-4},
}

@article {Gek91,
    AUTHOR = {Gekeler, Ernst-Ulrich},
     TITLE = {On finite {D}rinfeld{} modules},
   JOURNAL = {J. Algebra},
  FJOURNAL = {Journal of Algebra},
    VOLUME = {141},
      YEAR = {1991},
    NUMBER = {1},
     PAGES = {187--203},
      ISSN = {0021-8693,1090-266X},
   MRCLASS = {11G09 (11R58)},
  MRNUMBER = {1118323},
MRREVIEWER = {David\ Goss},
       DOI = {10.1016/0021-8693(91)90211-P},
       URL = {https://doi.org/10.1016/0021-8693(91)90211-P},
}

@misc{C22,
      title={CM Elliptic Curves: Volcanoes, Reality and Applications}, 
      author={Pete L. Clark},
      year={2022},
      eprint={2212.13316},
      archivePrefix={arXiv},
      primaryClass={math.NT},
      url={https://arxiv.org/abs/2212.13316}, 
}

@article {BR16,
    AUTHOR = {Breuer, Florian and R\"uck, Hans-Georg},
     TITLE = {Drinfeld modular polynomials in higher rank {II}: {K}ronecker
              congruences},
   JOURNAL = {J. Number Theory},
  FJOURNAL = {Journal of Number Theory},
    VOLUME = {165},
      YEAR = {2016},
     PAGES = {1--14},
      ISSN = {0022-314X,1096-1658},
   MRCLASS = {11G09 (11F52)},
  MRNUMBER = {3479213},
MRREVIEWER = {David\ Tweedle},
       DOI = {10.1016/j.jnt.2016.01.001},
       URL = {https://doi.org/10.1016/j.jnt.2016.01.001},
}

@misc{BR24,
      title={Drinfeld modular polynomials of level $T$}, 
      author={Florian Breuer and Mahefason Heriniaina Razafinjatovo},
      year={2024},
      eprint={2412.11324},
      archivePrefix={arXiv},
      primaryClass={math.NT},
      url={https://arxiv.org/abs/2412.11324}, 
}

@article {S95,
    AUTHOR = {Schweizer, Andreas},
     TITLE = {On the {D}rinfeld{} modular polynomial
              {$\Phi_T(X,Y)$}},
   JOURNAL = {J. Number Theory},
  FJOURNAL = {Journal of Number Theory},
    VOLUME = {52},
      YEAR = {1995},
    NUMBER = {1},
     PAGES = {53--68},
      ISSN = {0022-314X,1096-1658},
   MRCLASS = {11G09},
  MRNUMBER = {1331765},
       DOI = {10.1006/jnth.1995.1055},
       URL = {https://doi.org/10.1006/jnth.1995.1055},
}

@book {K96,
    AUTHOR = {Kohel, David Russell},
     TITLE = {Endomorphism rings of elliptic curves over finite fields},
      NOTE = {Thesis (Ph.D.)--University of California, Berkeley},
 PUBLISHER = {ProQuest LLC, Ann Arbor, MI},
      YEAR = {1996},
     PAGES = {117},
      ISBN = {978-0591-32123-4},
   MRCLASS = {99-05},
  MRNUMBER = {2695524},
       URL =
              {http://gateway.proquest.com/openurl?url_ver=Z39.88-2004&rft_val_fmt=info:ofi/fmt:kev:mtx:dissertation&res_dat=xri:pqdiss&rft_dat=xri:pqdiss:9723065},
}

@article {B97,
    AUTHOR = {Bae, Sunghan and Lee, Seungjae},
     TITLE = {On the coefficients of the {D}rinfeld modular equation},
   JOURNAL = {J. Number Theory},
  FJOURNAL = {Journal of Number Theory},
    VOLUME = {66},
      YEAR = {1997},
    NUMBER = {1},
     PAGES = {85--101},
      ISSN = {0022-314X,1096-1658},
   MRCLASS = {11G09 (11G15)},
  MRNUMBER = {1467191},
MRREVIEWER = {Yoshinori\ Hamahata},
       DOI = {10.1006/jnth.1997.2163},
       URL = {https://doi.org/10.1006/jnth.1997.2163},
}

@article {BJW17,
    AUTHOR = {Brooks, Ernest Hunter and Jetchev, Dimitar and Wesolowski,
              Benjamin},
     TITLE = {Isogeny graphs of ordinary abelian varieties},
   JOURNAL = {Res. Number Theory},
  FJOURNAL = {Research in Number Theory},
    VOLUME = {3},
      YEAR = {2017},
     PAGES = {Paper No. 28, 38},
      ISSN = {2522-0160,2363-9555},
   MRCLASS = {11G10},
  MRNUMBER = {3718564},
MRREVIEWER = {Davide\ Lombardo},
       DOI = {10.1007/s40993-017-0087-5},
       URL = {https://doi.org/10.1007/s40993-017-0087-5},
}

@misc{AMS25,
      title={Isogeny graphs of abelian varieties and singular ideals in orders}, 
      author={Sarah Arpin and Stefano Marseglia and Caleb Springer},
      year={2025},
      eprint={2508.03570},
      archivePrefix={arXiv},
      primaryClass={math.NT},
      url={https://arxiv.org/abs/2508.03570}, 
}

@incollection {AA10,
    AUTHOR = {Enge, Andreas and Sutherland, Andrew V.},
     TITLE = {Class invariants by the {CRT} method},
 BOOKTITLE = {Algorithmic number theory},
    SERIES = {Lecture Notes in Comput. Sci.},
    VOLUME = {6197},
     PAGES = {142--156},
 PUBLISHER = {Springer, Berlin},
      YEAR = {2010},
      ISBN = {978-3-642-14517-9; 3-642-14517-5},
   MRCLASS = {11Y16 (11G16 11G20)},
  MRNUMBER = {2721418},
MRREVIEWER = {A.\ Peth\H o},
       DOI = {10.1007/978-3-642-14518-6\_14},
       URL = {https://doi.org/10.1007/978-3-642-14518-6_14},
}

@book {C18,
    AUTHOR = {Caranay, P.},
     TITLE = {Computing Isogeny Volcanoes of Rank Two Drinfeld Modules},
      NOTE = {Thesis (Ph.D.)--University of Calgary, Calgary, Canada},
 PUBLISHER = {},
      YEAR = {2018},
     PAGES = {},
      ISBN = {},
   MRCLASS = {},
  MRNUMBER = {},
       URL =
              {http://hdl.handle.net/1880/106320},
}

@article {P98,
    AUTHOR = {Potemine, Igor Yu.},
     TITLE = {Minimal terminal {${\bf Q}$}-factorial models of {D}rinfeld
              coarse moduli schemes},
   JOURNAL = {Math. Phys. Anal. Geom.},
  FJOURNAL = {Mathematical Physics, Analysis and Geometry. An International
              Journal Devoted to the Theory and Applications of Analysis and
              Geometry to Physics},
    VOLUME = {1},
      YEAR = {1998},
    NUMBER = {2},
     PAGES = {171--191},
      ISSN = {1385-0172},
   MRCLASS = {11G09 (14D22 14M25)},
  MRNUMBER = {1690499},
MRREVIEWER = {Ernst-Ulrich Gekeler},
       DOI = {10.1023/A:1009724323513},
       URL = {https://doi.org/10.1023/A:1009724323513},
}

\end{document}